\documentclass[11pt,a4paper]{article}

\usepackage[T1]{fontenc}
\usepackage[utf8]{inputenc}
\usepackage{lmodern}
\usepackage[margin=1in]{geometry}
\usepackage{microtype}

\RequirePackage{amsthm,amsmath,amsfonts,amssymb,mathtools}
\RequirePackage{aliascnt,enumitem}
\RequirePackage{xcolor}
\RequirePackage[numbers,sort&compress]{natbib}
\RequirePackage[colorlinks=true,linkcolor=blue,citecolor=blue,urlcolor=blue]{hyperref}
\RequirePackage[nameinlink,capitalize]{cleveref}

\newcommand{\E}{\mathbb E}
\newcommand{\Pp}{\mathbb P}
\newcommand{\R}{\mathbb R}
\newcommand{\Var}{\operatorname{Var}}
\newcommand{\Cov}{\operatorname{Cov}}
\newcommand{\1}{\mathbf 1}
\newcommand{\dd}{\,d}
\newcommand{\ii}{\mathrm i}
\newcommand{\norm}[1]{\left\lVert #1\right\rVert}

\newcommand{\cF}{\mathcal F}
\newcommand{\doiurl}[1]{\href{https://doi.org/#1}{\nolinkurl{#1}}}

\theoremstyle{plain}
\newtheorem{theorem}{Theorem}[section]
\newaliascnt{proposition}{theorem}
\newtheorem{proposition}[proposition]{Proposition}
\aliascntresetthe{proposition}
\newaliascnt{lemma}{theorem}
\newtheorem{lemma}[lemma]{Lemma}
\aliascntresetthe{lemma}
\newaliascnt{corollary}{theorem}
\newtheorem{corollary}[corollary]{Corollary}
\aliascntresetthe{corollary}
\theoremstyle{definition}
\newaliascnt{remark}{theorem}

\aliascntresetthe{remark}
\numberwithin{equation}{section}

\title{Optimal normal approximation for symmetric statistics}
\author{%
Bing-Yi~Jing\thanks{School of Artificial Intelligence, The Chinese University of Hong Kong, Shenzhen, Shenzhen 518272, China. Email: \href{mailto:bingyijing@cuhk.edu.cn}{\texttt{bingyijing@cuhk.edu.cn}}.}
\and
Yiming~Liu\thanks{School of Economics, Jinan University, Guangzhou 510632, China. Email: \href{mailto:liuyiming@jnu.edu.cn}{\texttt{liuyiming@jnu.edu.cn}}.}
\and
Shaochen~Wang\thanks{School of Mathematics, South China University of Technology, Guangzhou 510640, China. Email: \href{mailto:mascwang@scut.edu.cn}{\texttt{mascwang@scut.edu.cn}}.}
\and
Wang~Zhou\thanks{Department of Statistics and Data Science, National University of Singapore, Singapore 119077, Singapore. Email: \href{mailto:stazw@nus.edu.sg}{\texttt{stazw@nus.edu.sg}}.}}
\date{}
\hypersetup{
  pdftitle={Optimal normal approximation for symmetric statistics},
  pdfauthor={Bing-Yi Jing; Liu Yiming; Wang Shaochen; Zhou Wang},
  pdfsubject={Optimal normal approximation}
}

\begin{document}
\maketitle

\begin{abstract}
Rates of convergence in normal approximation are fundamental to probability and statistics. The theory has evolved from normalized sums to Studentized statistics, smooth functions of sample means, and $U$-statistics, and more generally to symmetric statistics. A central question throughout this development has been to identify the weakest assumptions under which optimal rates of convergence can be established.
First, in this paper we address this question by establishing optimal bounds for adjusted normal approximations of symmetric statistics based on independent and identically distributed observations. More precisely, we bound the Kolmogorov error by a very neat sum of the classical third-moment term and the variance of the remainder term. This resolves a long-standing conjecture posed by Bentkus, Jing and Zhou (\textit{Ann. Probab.} \textbf{37} (2009), 2174--2199).
Secondly, we establish the optimal Berry-Esseen rate for jackknife Studentized \(U\)-statistics of order two under minimal conditions, i.e., a finite third absolute moment of the first projection and a finite \(5/3\)-absolute moment of the canonical kernel.  This resolves a long-standing conjecture in classical probability theory. Our proofs combine signed Prawitz smoothing with a conditional-variance submartingale, polynomial approximation, and a bounded change of measure.
\end{abstract}

\begin{center}
\begin{minipage}{0.9\textwidth}

\texorpdfstring{$U$-statistics}{U-statistics}.
\end{minipage}
\end{center}

\section{Introduction and main results}

\subsection{Introduction}

The relation between moment conditions and rates of normal approximation is a classical question in probability.  The Berry--Esseen theorem
bounds the error in the normal approximation to a sum of independent random variables by a normalized third absolute moment
\cite{Berry1941,Esseen1945}.  For identically distributed summands with a finite third moment, this gives the rate $n^{-1/2}$, where $n$ is the number of observations.  We ask whether this rate extends to nonlinear statistics and remains valid when the variance is estimated from the sample.  In both cases, one must control additional dependent terms.

For a centered symmetric statistic $S$, Hoeffding's decomposition
\cite{Hoeffding1948} separates the first projection from the orthogonal
higher-order components.  The first projection is a sum of independent
random variables and gives the usual linear approximation to $S$.
The sum of the higher-order components is the remainder.  Each component
is a sum of canonical kernels, whose integral in any one coordinate is
zero.  Normal approximation for $S$ then requires a bound on the effect
of the remainder.  Precise definitions are given in
Section~\ref{sec:main-results}.

For $U$-statistics of order two, Callaert and Janssen
\cite{CallaertJanssen1978} obtained a Berry--Esseen bound assuming a finite third
moment of the kernel.  General bounds for symmetric statistics were
proved by van Zwet \cite{vanZwet1984} and Friedrich \cite{Friedrich1989}.
For order-two $U$-statistics normalized by a nonrandom variance, the rate
$n^{-1/2}$ holds assuming a finite third moment of the first projection and
a finite $5/3$ moment of the canonical kernel; see Korolyuk and Borovskikh
\cite{KorolyukBorovskikh1986} and Friedrich \cite{Friedrich1989}.
The first projection is the conditional mean of the centered kernel
given a single observation; see \eqref{eq:student-projections}.
Bentkus, G\"otze and Zitikis \cite{BentkusGotzeZitikis1994} showed that
these exponents cannot be lowered in a uniform bound.
Bentkus and G\"otze \cite{BentkusGotze1996} treated higher orders, and
Borovskikh \cite{Borovskikh2001} obtained a bound using a weak $5/3$ moment.

The normal approximation can be refined using mixed moments of the
linear term and the remainder.  For order two, Alberink and Bentkus
\cite{AlberinkBentkus1,AlberinkBentkus2} proved that the error of this
adjusted normal approximation is bounded by the classical third-moment
term plus the variance of the remainder.  Bentkus, Jing and Zhou
\cite{BJZ2009} extended the adjusted approximation to every fixed order.
Their variance bound contains a squared logarithmic factor, which they
conjectured could be removed.  They also proved a log-free bound using
kernel moments of order strictly between one and two, and a bound for
smooth test functions that is linear in the remainder variance.
The question is whether this linear variance bound also holds in the
Kolmogorov distance, defined as the supremum of the absolute difference
between the distribution functions.

A second question concerns random normalization.  Arvesen
\cite{Arvesen1969} studied jackknife variance estimation for
$U$-statistics.  Berry--Esseen bounds for Studentized $U$-statistics were
proved by Callaert and Veraverbeke \cite{CallaertVeraverbeke1981}, Helmers
\cite{Helmers1985}, and Jing, Wang and Zhao \cite{JingWangZhao2000}.
For the Studentized sample mean, Bentkus and G\"otze
\cite{BentkusGotzeStudent1996} proved the classical third-moment bound.
For complete order-two $U$-statistics, Bentkus, Jing and Zhou
\cite[Theorem~4]{BJZ2009} obtained the rate $n^{-1/2}$ assuming a finite third
moment of the first projection and a finite moment of order $5/3+\varepsilon$
of the canonical kernel, for every fixed $\varepsilon>0$.  They conjectured
that $5/3$ suffices.  This exponent also appears in weak convergence
of Studentized $U$-processes \cite{Nasari2009}.

Concentration inequalities and Stein's method give another approach to
nonlinear normal approximation \cite{ChenShao2007,ShaoZhangZhou2016}.
Leung, Shao and Zhang \cite{LeungShaoZhang2024} obtained bounds for
Studentized nonlinear statistics with explicit dependence on the kernel
order.  Leung and Shao \cite{LeungShao2024} proved nonuniform bounds.
Their applications to $U$-statistics assume a finite third moment of the
full kernel.  Leung \cite{Leung2026} also studied incomplete $U$-statistics
with Bernoulli sampling.  In that setting, the error also depends on the
number of sampled terms.

We resolve the two questions above.  First, we remove the logarithmic
factor for every fixed Hoeffding order.  This closes a long-standing conjecture in classical probability theory.
The remainder enters the bound
through the sum of the variances of its orthogonal components.  Second,
we use this estimate to prove the $5/3$ moment bound for Studentized
order-two $U$-statistics.  The two results are linked by a polynomial
approximation to the Studentized statistic.  The linear dependence on the
remainder variance allows us to preserve the classical rate in this step.

\subsection{Main results}\label{sec:main-results}

Let $X_1,\ldots,X_n$ be independent observations with common law $P$ on
a measurable space $\mathcal X$.  Unless otherwise stated, sums over
observation indices range from $1$ to $n$.
For a real random variable $Y$, write $F_Y(x)=\Pp(Y\le x)$.
For a bounded function $H:\R\to\R$, write
$\norm H_\infty=\sup_{x\in\R}|H(x)|$.
Let $\Phi$ be the standard normal distribution function and
$\phi=\Phi'$ its density.  The symbols $C$ and $C_k$ denote positive
constants; $C$ is an absolute constant and $C_k$ depends only on $k$.
Their values may change from line to line.

Fix an integer $k\ge2$ and suppose $n\ge k$.
Let $S=S(X_1,\ldots,X_n)$ be a centered square-integrable symmetric
statistic of Hoeffding order at most $k$.  Its Hoeffding decomposition
separates the first projection, $\sum_{j=1}^n\E(S\mid X_j)$, from
the canonical components of orders two through $k$.
The first projection is the usual linear approximation to $S$.
We state the bound for a centered linear approximation with summands
$L_j=\ell(X_j)$, where $\ell:\mathcal X\to\R$ is measurable and
\begin{equation}
 \E L_1=0,\qquad n\E L_1^2=1,\qquad \E|L_1|^3<\infty.
\end{equation}
This gives the decomposition
\begin{equation}
 S=L+T,\qquad L=\sum_{j=1}^nL_j,
\end{equation}
where $T=S-L$ is the remainder.  Its Hoeffding decomposition is
\begin{equation}\label{eq:T-decomp}
 T=\sum_{r=1}^k T^{[r]},\qquad
 T^{[r]}=\sum_{|A|=r}T_A,\qquad
 T_A=t_{|A|}(X_i:i\in A).
\end{equation}
Here $A$ ranges over subsets of $\{1,\ldots,n\}$, $|A|$ is its
cardinality, and $t_r:\mathcal X^r\to\R$ is a measurable symmetric
kernel with $\E t_r(X_1,\ldots,X_r)^2<\infty$.
Each kernel is canonical in every coordinate:
\begin{equation}
 \E(T_A\mid X_i:i\in A\setminus\{j\})=0,\qquad j\in A.
\end{equation}
The order-one component of $T$ is $\sum_{j=1}^n\E(S\mid X_j)-L$.
This component accounts for the linear corrections used in Studentization.

Define
\begin{equation}
 \beta=n\E|L_1|^3,\qquad \gamma=\Var T.
\end{equation}
Write $T_{1\cdots r}=T_{\{1,\ldots,r\}}$ and define the mixed-moment
coefficients
\begin{equation}\label{eq:kappa-definition}
 \kappa_r=\binom nr\E(L_1\cdots L_rT_{1\cdots r}),
 \qquad 1\le r\le k.
\end{equation}
These expectations are finite by the Cauchy--Schwarz inequality; see
Section~\ref{sec:cf-estimates}.
With $\Phi^{(j)}$ denoting the $j$th derivative of $\Phi$, define
\begin{equation}\label{eq:Nk-definition}
 N_k(x)=\Phi(x)+\sum_{r=1}^k(-1)^{r+1}\kappa_r\Phi^{(r+1)}(x).
\end{equation}
This is the distribution function of a signed measure of total mass one.
In this notation, the bound of Bentkus, Jing and Zhou
\cite[Theorem~1]{BJZ2009} is
\begin{equation}\label{eq:BJZ-bound-intro}
 \norm{F_S-N_k}_\infty
 \le C_k\{\beta+\gamma\log^2(3+\gamma^{-1})\},
\end{equation}
where the second term is interpreted as zero when $\gamma=0$.
Our first result removes the logarithm.

\begin{theorem}\label{thm:main}
Under the preceding assumptions,
there is a constant $C_k<\infty$, depending only on $k$, such that
\begin{equation}\label{eq:main}
 \norm{F_S-N_k}_\infty\le C_k(\beta+\gamma).
\end{equation}
\end{theorem}

The distribution and the kernels may depend on $n$.
The Hoeffding components in \eqref{eq:T-decomp} are orthogonal, so
\begin{equation}
 \gamma=\sum_{r=1}^k\Var(T^{[r]}),\qquad
 \norm{F_S-N_k}_\infty
 \le C_k\left\{\beta+\sum_{r=1}^k\Var(T^{[r]})\right\}.
\end{equation}
Thus the remainder contribution is additive in the variances of its
orthogonal components.
The bound separates the classical third-moment term for $L$ from the
variance contributed by each order of the remainder.

Both terms are of optimal order; see Proposition~\ref{prop:main-sharp}.
In the example establishing optimality of the $\gamma$ term, the remainder
is canonical of order two, all adjustment coefficients vanish, and
$\beta/\gamma\to0$.
In general, the adjustment can be as large as a constant times
$\sqrt\gamma$ \cite[Theorem~2]{BJZ2009}.  Boundedness of the Gaussian
derivatives in \eqref{eq:Nk-definition} gives the standard normal
approximation bound
\begin{equation}\label{eq:ordinary-with-coefficients}
 \norm{F_S-\Phi}_\infty
 \le C_k\left\{\beta+\gamma+\sum_{r=1}^k|\kappa_r|\right\}.
\end{equation}
Hence an error of order $\beta+\gamma$ for the standard normal
approximation also requires control of the adjustment coefficients.

We now apply the variance bound to Studentization.  Expanding the
Studentized statistic gives symmetric remainders of finite Hoeffding
order, possibly with an order-one component.  Theorem~\ref{thm:main}
applies to these polynomial approximations and gives our second result.

Let $h:\mathcal X^2\to\R$ be a measurable symmetric kernel with
$\E|h(X_1,X_2)|<\infty$.  Define
\begin{equation}
 U_n=\binom n2^{-1}\sum_{i<j}h(X_i,X_j),\qquad
 \theta=\E h(X_1,X_2).
\end{equation}
Its first projection $g$ and canonical kernel $\eta$ are
\begin{equation}\label{eq:student-projections}
 \begin{split}
 g(x)&=\int h(x,y)\,P(\dd y)-\theta,\\
 \eta(x,y)&=h(x,y)-\theta-g(x)-g(y).
 \end{split}
\end{equation}
Write $g_i=g(X_i)$ and $\eta_{ij}=\eta(X_i,X_j)$.
Then $h(X_i,X_j)-\theta=g_i+g_j+\eta_{ij}$, $\E g_1=0$, and
$\int\eta(x,y)\,P(\dd y)=0$ for $P$-almost every $x$.
The same identity holds after interchanging the two coordinates.
Assume $0<\sigma_g^2=\E g_1^2<\infty$.
For $n\ge4$, define the jackknife variance estimator by
\begin{equation}
 \widehat\sigma_g^2
 =\frac{n-1}{(n-2)^2}\sum_{i=1}^n
 \left\{\frac1{n-1}\sum_{j\ne i}h(X_i,X_j)-U_n\right\}^2,
\end{equation}
and let $\widehat\sigma_g$ be its nonnegative square root.
The Studentized statistic is
\begin{equation}
 \widehat S_n=\frac{\sqrt n\,(U_n-\theta)}{2\widehat\sigma_g}.
\end{equation}
Set $\widehat S_n=0$ when $\widehat\sigma_g=0$.
The following bound also holds under any other fixed convention for this event.

\begin{theorem}\label{thm:studentized}
For the kernel $h$ and statistic $\widehat S_n$ defined above, there is an
absolute constant $C$ such that, for $n\ge4$,
\begin{equation}\label{eq:studentized-main}
 \norm{F_{\widehat S_n}-\Phi}_\infty
 \le\frac{C}{\sqrt n}\left\{
 \frac{\E|g_1|^3}{\sigma_g^3}
 +\frac{\E|\eta_{12}|^{5/3}}{\sigma_g^{5/3}}\right\}.
\end{equation}
\end{theorem}

This establishes the endpoint conjectured in
\cite[Theorem~4 and the preceding discussion]{BJZ2009}.
Proposition~\ref{prop:student-sharp} shows that the rate and both moment
exponents are optimal for bounds uniform over the distribution and kernel.
The examples use distributions and kernels that depend on $n$.

\subsection{Notation, proof strategy and organization}

We use the notation of Section~\ref{sec:main-results} throughout.
Within expectations, omitted arguments of one- and two-variable kernels
are taken to be $X_1$ and $(X_1,X_2)$, respectively, under the indicated law.
For $r\ge1$, let $P^r=P^{\otimes r}$ be the $r$-fold product law.
For kernels on $\mathcal X^r$, $\norm{\cdot}_\infty$ denotes the essential
supremum with respect to $P^r$.  The symbol $\1_A$ denotes the indicator
of an event or measurable set $A$.  Subscripts on constants indicate
dependence on the listed fixed parameters.
We write $a_n=O(b_n)$ if $|a_n|\le Cb_n$ for all sufficiently large $n$,
$a_n=o(b_n)$ if $a_n/b_n\to0$, and $a_n\sim b_n$ if $a_n/b_n\to1$,
with $b_n>0$.  For positive sequences, $a_n\asymp b_n$ means that
$a_n/b_n$ is bounded above and below by positive constants.
The same notation is used for functions in a specified limiting range.

For the component variances, define
\begin{equation}
 \gamma_r=\binom nr\E T_{1\cdots r}^2=\Var(T^{[r]}),
 \qquad \gamma=\sum_{r=1}^k\gamma_r.
\end{equation}
Here $T_A$ denotes a single kernel and $T^{[r]}$ the sum of these kernels
over all index sets of size $r$.  The mixed moments of $L$ and $T$ enter
the approximation through the coefficients \eqref{eq:kappa-definition}.

We first explain how to remove the logarithm.  Let $t$ denote the
Fourier variable and let $m$ be the number of observations in a block.
The expansion in \cite{BJZ2009} has a Taylor remainder bounded by
\begin{equation}\label{eq:old-remainder-intro}
 C_k\gamma t^2m/n.
\end{equation}
At large frequencies, the choice $m\asymp nt^{-2}\log|t|$ gives the
required characteristic-function decay.  Integrating the absolute value of
\eqref{eq:old-remainder-intro} against $\dd t/|t|$ then produces a
squared logarithm.  We instead use the signed smoothing inequality of
Prawitz \cite{Prawitz1972}, in the form stated by Pinelis
\cite{Pinelis2013}.  The running maximum of a conditional-variance
submartingale gives nested exceptional events with probability $O(\gamma)$.
For each sample point, the exceptional frequency set is empty or, up to
endpoints, has the form $r<|t|<a$, with inner radius $r$ and outer radius $a$.
The smoothing multiplier has a uniformly bounded oscillatory integral over
such sets.  Applying this bound before taking expectations controls the
exceptional term with phase $e^{\ii tS}$; the other exceptional terms are
integrated in absolute value.  On the complementary event, conditional
independence gives the required decay.  Terms involving at least two
block indices have smaller variance and can also be integrated in absolute value.

For Studentization, normalize to $\sigma_g=1$ and define
$\rho=n^{-1/2}(\E|g_1|^3+\E|\eta_{12}|^{5/3})$.
Truncation affects the statistic only on an event of probability $O(\rho)$.
A cutoff on the conditional $5/3$ moment of the canonical kernel gives
uniform row second-moment bounds.  Outside a further event of probability
$O(\rho)$, the truncated ratio lies between two polynomials of Hoeffding
order at most ten.  For each polynomial, the third-moment term and the
remainder variance are $O(\rho)$.  Rosenthal's inequality controls the
truncated pair sum and the vector of its row sums.

It remains to control the adjustment coefficients in the Studentization argument.  We do this through signed mixed moments.
The adjustment coefficients require a bound on signed mixed moments.
Let $u$ be a linear summand with $\E u(X_1)=0$, $n\E u(X_1)^2=1$, and
$\norm u_\infty\le K$, where $K\ge1$ is fixed.  Consider
$P_t(\dd x)=(1+t u(x))P(\dd x)$ for $|t|\le(2K)^{-1}$, and write
$\E_t$ for expectation under the corresponding product law.  For a
symmetric remainder $R$ of fixed Hoeffding order, $\E_tR$ is a polynomial
in $t$.  Its constant coefficient is $\E R$, and its other coefficients
are the adjustment coefficients.  A uniform bound on $|\E_tR|$ bounds
them all; see Corollary~\ref{cor:finite-order}.  For the two polynomial
approximations, these bounds are $O(\rho)$, so
\eqref{eq:ordinary-with-coefficients} gives the desired rate.
Using \eqref{eq:BJZ-bound-intro} with the same estimates would instead
give $O\{\rho\log^2(3+\rho^{-1})\}$.

Section~\ref{sec:prelim} gives the characteristic-function estimates,
the signed smoothing inequality, and the finite-order consequence used
for Studentization.  Section~\ref{sec:proof} proves
Theorem~\ref{thm:main} and its optimality.  Section~\ref{sec:studentized}
proves Theorem~\ref{thm:studentized} and the corresponding optimality
statements.

\section{Preliminary estimates}\label{sec:prelim}

We collect the characteristic-function and smoothing estimates, together
with a consequence of the main theorem that will be used for Studentization.

\subsection{Characteristic functions and coefficient bounds}\label{sec:cf-estimates}

We work under the assumptions of Theorem~\ref{thm:main}.
Canonicality gives the orthogonality relation
\[
 \E(T_AT_B)=0\qquad(A\ne B).
\]
Canonicality also gives
\begin{equation}\label{eq:kappa-with-T}
  \E(L_1\cdots L_rT)
  =\E(L_1\cdots L_rT_{\{1,\dots,r\}}).
\end{equation}
To verify \eqref{eq:kappa-with-T}, consider a component $T_A$ of $T$.  If
$A$ contains an index outside $\{1,\dots,r\}$, conditioning on all variables
except the one with that index gives a zero contribution by canonicality.
If $A\subsetneq\{1,\dots,r\}$, choose $j\in\{1,\dots,r\}\setminus A$;
then the contribution is zero because $L_j$ is centered and independent of
the remaining factors.  Thus only $A=\{1,\dots,r\}$ contributes.

The expectation in \eqref{eq:kappa-definition} is finite by the Cauchy--Schwarz
inequality and the independence of $L_1,\dots,L_r$:
\[
 \E|L_1\cdots L_rT_{\{1,\dots,r\}}|
 \le \{\E(L_1^2\cdots L_r^2)\}^{1/2}
      \{\E T_{\{1,\dots,r\}}^2\}^{1/2}
 =n^{-r/2}\{\E T_{\{1,\dots,r\}}^2\}^{1/2}.
\]
The coefficient normalization in \eqref{eq:Nk-definition} agrees with
the characteristic-function expansion in \cite{BJZ2009}.
Let $\ii=\sqrt{-1}$.  Write
\begin{equation}
  f(t)=\E e^{\ii tS},
  \qquad
  g_k(t)=\int_{\R}e^{\ii tx}\,\dd N_k(x).
\end{equation}
Since $\Phi^{(r+2)}=\phi^{(r+1)}$, where $\phi=\Phi'$, repeated integration
by parts gives
\begin{equation}\label{eq:g-explicit}
  g_k(t)=e^{-t^2/2}\left\{1+\sum_{r=1}^k\kappa_r(\ii t)^{r+1}\right\}.
\end{equation}
In particular, if $|\kappa_r|\le1$ for all $r$, then
\begin{equation}\label{eq:g-L1}
  \int_{\R}|g_k(t)|\,\dd t\le C_k.
\end{equation}
We use the following estimates from \cite{BJZ2009}.

\begin{lemma}\label{lem:BJZ-inputs}
There are constants depending only on $k$ such that the following hold.
\begin{enumerate}[label=\textup{(\roman*)}]
\item\label{item:basic}
  $1\le\sqrt n\,\beta$, and the adjustment coefficients satisfy
  $|\kappa_r|^2\le\gamma_r\le\gamma$.
\item\label{item:low}
  For every $t\in\R$,
  \begin{equation}\label{eq:BJZ-low}
    |f(t)-g_k(t)|
    \le C_k\bigl(|t|+|t|^{k+4}\bigr)(\beta+\gamma).
  \end{equation}
\item\label{item:prop2}
  Let $M=\{1,\dots,m\}$, let $T^{(0)}$ be the sum of the kernels avoiding
  $M$, and let $T^{(m)}=T-T^{(0)}$.  Define
  \begin{equation}\label{eq:f0-f1-f4}
    f_0(t)=\E e^{\ii t(L+T^{(0)})},\qquad
    f_1(t)=\ii t\,\E\{e^{\ii t(L+T^{(0)})}T^{(m)}\},
  \end{equation}
  and
  \begin{equation}\label{eq:f4}
    f_4(t)=\E e^{\ii tL}+\ii t\,\E(e^{\ii tL}T).
  \end{equation}
  If $m$ is even and $2k\le m\le n$, then
  \begin{equation}\label{eq:BJZ-prop2}
    |f_0(t)+f_1(t)-f_4(t)|
    \le C_k\gamma t^2|\vartheta(t)|^{m/2},
  \end{equation}
  where $\vartheta(t)=\E e^{\ii tL_1}$.
\item\label{item:prop3}
  Let $\rho(t)=e^{-t^2/(2n)}$.  If $\beta\le1$, $\gamma\le1$,
  $n\ge4k$, $1\le m\le n/4$, and $|t|\ge1$, then
  \begin{equation}\label{eq:BJZ-prop3}
    |f_4(t)-g_k(t)|
    \le C_k(\beta+\gamma)|t|^{k+4}
       \bigl(|\vartheta(t)|^m+\rho(t)^m\bigr).
  \end{equation}
\end{enumerate}
\end{lemma}

\begin{proof}
For item \ref{item:basic}, Lyapunov's inequality and
$n\E L_1^2=1$ give
\[
  1=(n\E L_1^2)^{3/2}
  \le n^{3/2}\E|L_1|^3=\sqrt n\,\beta.
\]
Moreover, by \eqref{eq:kappa-definition} and the Cauchy--Schwarz estimate
above,
\begin{align*}
 |\kappa_r|^2
 &\le \binom nr^2n^{-r}\E T_{\{1,\dots,r\}}^2=\frac{\binom nr}{n^r}\gamma_r
 \le\gamma_r.
\end{align*}
Items \ref{item:low}, \ref{item:prop2}, and \ref{item:prop3} are,
respectively, Corollary~1, Proposition~2, and Proposition~3 of
\cite{BJZ2009}.  In item \ref{item:prop3}, the polynomial is \eqref{eq:g-explicit},
as obtained in the proof of their Proposition~3.
\end{proof}

We use the following local Gaussian decay estimate for the characteristic
function of a linear summand.

\begin{lemma}\label{lem:theta}
Let $\vartheta(t)=\E e^{\ii tL_1}$.  Then
\begin{equation}\label{eq:theta-bound}
  |\vartheta(t)|\le e^{-t^2/(4n)}
  \qquad\text{whenever } |t|\beta\le3/8.
\end{equation}
\end{lemma}

\begin{proof}
Let $L_1'$ be an independent copy of $L_1$ and let $Y=L_1-L_1'$.  Then
\[
  |\vartheta(t)|^2=\E\cos(tY),\qquad
  \E Y^2=2/n,
\]
and $|x-y|^3\le4(|x|^3+|y|^3)$ gives
\[
  \E|Y|^3\le8\E|L_1|^3=8\beta/n.
\]
The Taylor inequality
$\cos u\le1-u^2/2+|u|^3/6$ therefore yields
\[
  |\vartheta(t)|^2
  \le1-\frac{t^2}{n}+\frac{4|t|^3\beta}{3n}.
\]
If $|t|\beta\le3/8$, the last term is at most $t^2/(2n)$, and hence
\[
  |\vartheta(t)|^2\le1-\frac{t^2}{2n}
  \le e^{-t^2/(2n)}.
\]
Moreover, \Cref{lem:BJZ-inputs}\ref{item:basic} gives
$|t|\le(3/8)\beta^{-1}\le(3/8)\sqrt n$, so
$t^2/(2n)\le9/128$ and the intermediate upper bound is positive.  Taking
square roots proves \eqref{eq:theta-bound}.
\end{proof}

\subsection{A signed Prawitz smoothing inequality}
The classical Esseen inequality \cite[Chapter~II]{Esseen1945} takes the absolute value of the Fourier
remainder before integration.  The cancellation used below requires a
one-sided smoothing formula in which the Fourier integral remains signed.
We use the bounding filter of Prawitz.

For $|u|<1$, define
\begin{equation}\label{eq:Prawitz-filter}
  \mathcal M(u)=\bigl[(1-|u|)\pi u\cot(\pi u)+|u|\bigr]
       -\ii(1-|u|)\pi u,
\end{equation}
with $\mathcal M(0)=1$, and set $\mathcal M(u)=0$ for $|u|\ge1$.  For $a>0$ and
$\sigma\in\{-1,1\}$, let
\begin{equation}
  \Lambda_{\sigma,a}(t)
  =\frac{\ii}{2\pi t}\mathcal M(\sigma t/a)\1_{\{0<|t|<a\}},\qquad t\ne0,
\end{equation}
and set $\Lambda_{\sigma,a}(0)=0$.
All integrals through the origin are understood as symmetric principal values,
denoted by $\operatorname{p.v.}$.
The next lemma gives the smoothing bound and a uniform estimate
over symmetric frequency intervals.

\begin{lemma}\label{lem:smoothing}
Let $F$ be a distribution function with Fourier--Stieltjes transform $f$.
Let $N_k$ and $g_k$ be given by \eqref{eq:Nk-definition} and
\eqref{eq:g-explicit}, and suppose that $|\kappa_r|\le1$ for
$1\le r\le k$.  Assume also that
$(f(t)-g_k(t))/t$ is locally integrable at the origin.  Then, for every
$a\ge1$,
\begin{equation}\label{eq:signed-smoothing}
  \norm{F-N_k}_\infty
  \le \frac{C_k}{a}
  +C\max_{\sigma=\pm1}\sup_{x\in\R}
      \left|\operatorname{p.v.}\int_{-a}^a
        e^{-\ii tx}\Lambda_{\sigma,a}(t)\{f(t)-g_k(t)\}\,\dd t\right|.
\end{equation}
Moreover,
\begin{equation}\label{eq:Lambda-size}
  |\Lambda_{\sigma,a}(t)|\le\frac C{|t|},
  \qquad
  |t\Lambda_{\sigma,a}(t)|\le C,
\end{equation}
and
\begin{equation}\label{eq:Lambda-osc}
  \sup_{z\in\R}\sup_{0\le r<a}
  \left|\int_{r<|t|<a}e^{\ii tz}\Lambda_{\sigma,a}(t)\,\dd t\right|
  \le C.
\end{equation}
\end{lemma}

\begin{proof}
For a function $\psi$ for which the principal value exists, write
\[
  \mathcal G(\psi)(x)=\frac{\ii}{2\pi}\operatorname{p.v.}
       \int_{\R}e^{-\ii tx}\frac{\psi(t)}t\,\dd t.
\]
The Prawitz bounding-filter theorem states that, for every distribution
function $F$, every $a>0$, and every $x\in\R$, with
$F(x-)=\lim_{y\uparrow x}F(y)$ and $F(x+)=F(x)$,
\begin{equation}\label{eq:Prawitz-one-sided}
 \frac12+\mathcal G\{\mathcal M(-\cdot/a)f\}(x)
 \le F(x-)
 \le F(x+)
 \le \frac12+\mathcal G\{\mathcal M(\cdot/a)f\}(x).
\end{equation}
This is the one-sided smoothing inequality of Prawitz
\cite{Prawitz1972}. The filter \eqref{eq:Prawitz-filter} and the
normalization in \eqref{eq:Prawitz-one-sided} are given in
\cite[Section~2]{Pinelis2013}.

The signed measure $\dd N_k$ has a smooth density obtained by
differentiating \eqref{eq:Nk-definition}, and its Fourier transform is the
integrable function \eqref{eq:g-explicit}.  Fourier inversion therefore gives
\begin{equation}\label{eq:N-inversion}
  N_k(x)=\frac12+\mathcal G(g_k)(x),\qquad x\in\R.
\end{equation}
Subtracting \eqref{eq:N-inversion} from the two bounds in
\eqref{eq:Prawitz-one-sided} yields
\begin{align*}
 \norm{F-N_k}_\infty
 &\le \max_{\sigma=\pm1}\sup_x
       |\mathcal G\{\mathcal M(\sigma\cdot/a)(f-g_k)\}(x)|\\
 &\qquad\qquad+\max_{\sigma=\pm1}\sup_x
       |\mathcal G\{(\mathcal M(\sigma\cdot/a)-1)g_k\}(x)|.
\end{align*}
The filter is bounded and continuous at $0$ and $\pm1$, and
$\pi u\cot(\pi u)=1+O(u^2)$ near zero.  Thus
\begin{equation}
 |\mathcal M(u)-1|\le C\min\{|u|,1\},\qquad u\in\R.
\end{equation}
Using \eqref{eq:g-L1}, we obtain, uniformly in $x$ and $\sigma$,
\begin{equation}
 \left|\mathcal G\{(\mathcal M(\sigma\cdot/a)-1)g_k\}(x)\right|
 \le\frac{C}{a}\int_\R|g_k(t)|\,\dd t\le\frac{C_k}{a}.
\end{equation}
This proves \eqref{eq:signed-smoothing}.  The assumed local integrability
ensures that the integral containing $f-g_k$ is well defined at zero.  In our
application, this follows from \eqref{eq:BJZ-low}.

It remains to verify the multiplier estimates.  Write $\mathcal M=\mathcal M_1+\ii \mathcal M_2$.
The real part $\mathcal M_1$ is even and bounded, and has bounded variation on $[-1,1]$.
Its total variation on this interval is denoted by
$\operatorname{Var}_{[-1,1]}(\mathcal M_1)$.
These properties follow from \eqref{eq:Prawitz-filter}; see also
\cite[Remark~2.1]{Pinelis2013}.  The imaginary part is
\[
  \mathcal M_2(u)=-\pi u(1-|u|)\1_{\{|u|<1\}},
\]
which is odd.  Boundedness of $\mathcal M$ gives \eqref{eq:Lambda-size} directly.
For the annular integral, the $\mathcal M_2$ contribution is bounded by
\[
  C\int_r^a\frac1a\left(1-\frac ta\right)\,\dd t\le C.
\]
The contribution of $\mathcal M_1$ over the symmetric set $r<|t|<a$ is a constant
multiple of
\[
  \int_r^a \mathcal M_1(t/a)\frac{\sin(tz)}t\,\dd t.
\]
The primitive $\int_r^y\sin(tz)\,\dd t/t$ is uniformly bounded in
$r,y,z$.  Integration by parts for functions of bounded variation bounds
the absolute value of the last integral by
$C\{\norm{\mathcal M_1}_\infty+\operatorname{Var}_{[-1,1]}(\mathcal M_1)\}$, uniformly in
$r,a,z$.  This proves \eqref{eq:Lambda-osc}.
\end{proof}

\subsection{A finite-order consequence}

The following consequence will be used for Studentization.  A bounded
change of measure controls the mean and all adjustment coefficients.

\begin{corollary}\label{cor:finite-order}
Let $n\ge d\ge2$ be integers, and let $X_1,\ldots,X_n$ be independent with common
law $P$.
Suppose that
\[
 U=\sum_{i=1}^n u(X_i),\qquad
 \E u(X_1)=0,\qquad n\E u(X_1)^2=1,\qquad \norm u_\infty\le K,
\]
where $K\ge1$.  Let $R$ be a square-integrable symmetric statistic of
Hoeffding order at most $d$, not necessarily centered.  For
$|t|\le t_0=(2K)^{-1}$, let $\E_t$ denote expectation under the product
law with common marginal distribution
\[
 P_t(\dd x)=(1+t u(x))P(\dd x),
 \qquad \varepsilon_R=\sup_{|t|\le t_0}|\E_tR|.
\]
Then
\begin{equation}\label{eq:finite-order-bound}
 \norm{F_{U+R}-\Phi}_\infty
 \le C_{d,K}\left\{n\E|u(X_1)|^3+\Var R+\varepsilon_R\right\}.
\end{equation}
\end{corollary}

\begin{proof}
The measure $P_t$ is a probability measure, and its density lies between
$1/2$ and $3/2$.  For $0\le s\le n$, define
\[
 H_s=\sum_{\substack{I\subset\{1,\ldots,n\}\\|I|=s}}
                \prod_{i\in I}u(X_i),\qquad H_0=1.
\]
The Hoeffding decomposition of $R$ contains terms depending on at most
$d$ observations.  For $s>d$, every product in $H_s$ has a centered
factor independent of each such term.  Thus $\E(RH_s)=0$, and
\begin{equation}
 \E_tR=\E\left\{R\prod_{i=1}^n(1+t u(X_i))\right\}
       =\sum_{s=0}^d t^s\E(RH_s).
\end{equation}
All expectations exist since $u$ is bounded and $R\in L^2$.
Evaluate this polynomial at $t_j=-t_0+2jt_0/d$, $0\le j\le d$.
The points are distinct, so the Vandermonde matrix
$(t_j^s)_{0\le j,s\le d}$ is invertible.  Its inverse depends
only on $d$ and $K$.  Consequently,
\begin{equation}\label{eq:finite-order-coefficients}
 \sum_{s=0}^d|\E(RH_s)|\le C_{d,K}\varepsilon_R.
\end{equation}
Apply \eqref{eq:ordinary-with-coefficients}, which follows from
Theorem~\ref{thm:main}, to $U+R-\E R$.  By symmetry and
\eqref{eq:kappa-with-T}, its adjustment coefficients are
\[
 \kappa_s=\binom ns\E\{u(X_1)\cdots u(X_s)R\}=\E(RH_s),
 \qquad 1\le s\le d.
\]
Here $\E H_s=0$ for $s\ge1$, so centering $R$ does not change these
coefficients.  The remainder variance is $\Var R$, and
$|\E R|\le\varepsilon_R$.  Shifting the approximating normal distribution
by $\E R$ changes its distribution function by at most $\norm{\Phi'}_\infty|\E R|$.
Equation~\eqref{eq:finite-order-coefficients} now proves
\eqref{eq:finite-order-bound}.
\end{proof}

\section{Proof and optimality of Theorem~\texorpdfstring{\ref{thm:main}}{1.1}}\label{sec:proof}

We bound the first Taylor remainder in the signed smoothing integral
and then use the subsequent expansions of \cite{BJZ2009}.  Lower bounds
follow in Section~\ref{sec:main-sharp}.

\subsection{Reduction and choice of block size}
The coefficient bound in Lemma~\ref{lem:BJZ-inputs} and boundedness of
the Gaussian derivatives give
\begin{equation}
 \norm{F_S-N_k}_\infty\le C_k(1+\sqrt\gamma).
\end{equation}
If $\beta+\gamma\ge c_k>0$, this is at most
$C_k(\beta+\gamma)$, after increasing $C_k$.  Indeed, use
$1+\sqrt\gamma\le2\gamma$ when $\gamma\ge1$ and
$1+\sqrt\gamma\le2$ otherwise.  It therefore suffices to assume
\begin{equation}\label{eq:small-case}
 \beta+\gamma\le c_k,
\end{equation}
where $c_k$ is chosen below.

Choose $A_k$ as in \eqref{eq:A-choice}, and then $t_*=t_*(k)>e$ so large that
\[
 A_k\log t_*\ge4k+2,\qquad A_k\log t_*/t_*^2\le1/4.
\]
Finally, choose $c_k<\min\{1,(16t_*)^{-1},(4k)^{-1/2}\}$.
Since $1\le\sqrt n\,\beta$, condition \eqref{eq:small-case} implies
$n\ge4k$.  Define
\begin{equation}
 a=\frac1{8(\beta+\gamma)}.
\end{equation}
Then
\begin{equation}\label{eq:a-range}
 a>2t_*,\qquad a\beta\le1/8,\qquad
 a\le\sqrt n/8,\qquad a^{-1}=8(\beta+\gamma).
\end{equation}
In particular, Lemma~\ref{lem:theta} applies for $|t|\le a$, and we have
$|\kappa_r|\le1$.  By Lemma~\ref{lem:smoothing}, it remains to prove, for
both signs $\sigma$,
\begin{equation}\label{eq:signed-target}
 \sup_x\left|\int_{-a}^a e^{-\ii tx}\Lambda_{\sigma,a}(t)
                       \{f(t)-g_k(t)\}\,\dd t\right|
 \le C_k(\beta+\gamma).
\end{equation}
The low frequencies are controlled by \eqref{eq:BJZ-low}:
\begin{equation}\label{eq:low-integral}
 \int_{|t|\le t_*}|\Lambda_{\sigma,a}(t)|\,|f(t)-g_k(t)|\,\dd t
 \le C_k(\beta+\gamma)\int_0^{t_*}(1+t^{k+3})\,\dd t
 \le C_k(\beta+\gamma).
\end{equation}

For $t_*\le|t|<a$, choose the even integer
\begin{equation}
  m(t)=2\left\lfloor
    \frac{A_kn|t|^{-2}\log|t|}{2}
  \right\rfloor,
  \qquad
  q(t)=n-m(t),
\end{equation}
where
\begin{equation}\label{eq:A-choice}
  A_k>\max\{64,16(k+5)\}
\end{equation}
is fixed.  By taking $t_*=t_*(k)$ sufficiently large, we have
\begin{equation}
  4k\le m(t)\le n/4
  \qquad(t_*<|t|<a).
\end{equation}
Indeed, $|t|\le a\le\sqrt n$ gives $n/t^2\ge1$, so the lower bound follows
from $A_k\log t_*\ge4k+2$.  The upper bound follows from
$A_k\log t_*/t_*^2\le1/4$.  Since $t\mapsto t^{-2}\log t$ is decreasing for
$t>e^{1/2}$, $m(t)$ is a measurable, nonincreasing step function and $q(t)$
is nondecreasing in $|t|$.
Furthermore, if $A_k\log t_*\ge4$, then
\begin{equation}\label{eq:mt-two-sided}
  \frac{A_k}{2}\log|t|
  \le \frac{t^2m(t)}n
  \le A_k\log|t|.
\end{equation}
By \Cref{lem:theta}, \eqref{eq:a-range}, and
\eqref{eq:mt-two-sided},
\begin{equation}\label{eq:damping-polynomial}
  |\vartheta(t)|^{m(t)}+\rho(t)^{m(t)}
  \le 2|t|^{-A_k/8}.
\end{equation}

For $M=M(t)=\{1,\dots,m(t)\}$, let $T^{(0)}$ and $T^{(m)}$ be as in
\Cref{lem:BJZ-inputs}\ref{item:prop2}, and define $f_0,f_1,f_4$ by
\eqref{eq:f0-f1-f4}--\eqref{eq:f4}.  Pointwise in $t$,
\begin{equation}\label{eq:three-way}
  f-g_k=(f-f_0-f_1)+(f_0+f_1-f_4)+(f_4-g_k).
\end{equation}
The bounds from \cite{BJZ2009} for the last two terms can be integrated
in absolute value.  Indeed,
\Cref{lem:theta} and \eqref{eq:mt-two-sided} give
\begin{equation}
  |\vartheta(t)|^{m(t)/2}\le |t|^{-A_k/16},
  \qquad
  |\vartheta(t)|^{m(t)}+\rho(t)^{m(t)}\le 2|t|^{-A_k/8}.
\end{equation}
Consequently, by \eqref{eq:BJZ-prop2} and \eqref{eq:Lambda-size},
\begin{align}
  \int_{t_*<|t|<a}|\Lambda_{\sigma,a}(t)|
       |f_0+f_1-f_4|(t)\,\dd t
  &\le C_k\gamma\int_{t_*}^\infty t^{1-A_k/16}\,\dd t
   \le C_k\gamma,
  \label{eq:prop2-integrated}
\end{align}
provided $A_k>32$.  Similarly, \eqref{eq:BJZ-prop3} yields
\begin{align}
  \int_{t_*<|t|<a}|\Lambda_{\sigma,a}(t)|
       |f_4-g_k|(t)\,\dd t
  &\le C_k(\beta+\gamma)
       \int_{t_*}^\infty t^{k+3-A_k/8}\,\dd t\\
  &\le C_k(\beta+\gamma),
  \label{eq:prop3-integrated}
\end{align}
by \eqref{eq:A-choice}.
It remains to establish
\begin{equation}\label{eq:integrated-P1}
  \sup_x\left|
    \int_{t_*<|t|<a}e^{-\ii tx}\Lambda_{\sigma,a}(t)
      \{f(t)-f_0(t)-f_1(t)\}\,\dd t
  \right|
  \le C_k\gamma.
\end{equation}

In the remainder of this section, $m=m(t)$, $q=n-m$, and
$M=\{1,\dots,m\}$.  For an index set $A$,
write $L_A=\sum_{j\in A}L_j$ and
$M^c=\{1,\ldots,n\}\setminus M$.

\subsection{Decomposition of the remainder}

Write
\begin{equation}
  T^{(m)}=U+V,
\end{equation}
where
\begin{align}
  U&=\sum_{r=1}^k\sum_{\substack{|A|=r\\|A\cap M|=1}}T_A,
  \\
  V&=\sum_{r=2}^k\sum_{\substack{|A|=r\\|A\cap M|\ge2}}T_A.
\end{align}
Orthogonality gives
\begin{equation}\label{eq:UV-variance}
  \E U^2\le C_k\gamma\frac mn,
  \qquad
  \E V^2\le C_k\gamma\left(\frac mn\right)^2.
\end{equation}
Indeed, if $A$ is uniformly distributed over the $r$-subsets of
$\{1,\dots,n\}$, then
\[
  \Pp(|A\cap M|=1)\le r\frac mn,
\]
and
\[
  \Pp(|A\cap M|\ge2)
  \le \E\binom{|A\cap M|}{2}
  =\binom r2\frac{m(m-1)}{n(n-1)}
  \le C_k\left(\frac mn\right)^2.
\]
Multiplying by $\gamma_r$ and summing over $r$ proves
\eqref{eq:UV-variance}.

Define $\Psi_t(w)=e^{\ii tw}-1-\ii tw$.  The definitions of $f_0$ and
$f_1$ give
\begin{equation}\label{eq:R-exact}
  f(t)-f_0(t)-f_1(t)
  =\E\left[e^{\ii t(L+T^{(0)})}\Psi_t(U+V)\right].
\end{equation}
For each $t$, insert $1=\1_{H_q}+\1_{B_q}$ in
\eqref{eq:R-exact}, where $H_q$ and $B_q$ are defined in
Lemma~\ref{lem:submartingale}.  We estimate the two terms separately.

\subsection{Conditional variances}

For $0\le q\le n-1$, define
\begin{equation}
  I_q=\{n-q+1,\dots,n\},
  \qquad
  \cF_q=\sigma(X_i:i\in I_q).
\end{equation}
Here $I_0=\varnothing$ and $\cF_0$ is the trivial sigma-field.
Let
\begin{equation}\label{eq:Uq}
  \mathcal U_q=\sum_{r=1}^k
       \sum_{\substack{B\subset I_q\\|B|=r-1}}
       T_{\{1\}\cup B},
\end{equation}
and set
\begin{equation}\label{eq:Zq}
  Z_q=n\E(\mathcal U_q^2\mid\cF_q).
\end{equation}
Because $\mathcal U_q$ is measurable with respect to $\sigma(X_1,X_i:i\in I_q)$, the
conditional expectation in \eqref{eq:Zq} integrates only over $X_1$.
We use the running maximum of this process to define the exceptional
events in the next lemma.

\begin{lemma}\label{lem:submartingale}
The process $(Z_q,\cF_q)_{0\le q\le n-1}$ is a nonnegative submartingale, and
\begin{equation}
  \E Z_{n-1}\le k\gamma.
\end{equation}
Consequently, for every $\eta>0$,
\begin{equation}
  H_q=\left\{\max_{0\le r\le q}Z_r\le\eta\right\},
  \qquad
  B_q=H_q^c,
\end{equation}
satisfy
\begin{equation}\label{eq:PB}
  \Pp(B_q)\le\eta^{-1}k\gamma.
\end{equation}
\end{lemma}

\begin{proof}
For $0\le q\le n-2$, let $i_{q+1}=n-q$, so $I_{q+1}=I_q\cup\{i_{q+1}\}$.  Decompose
\begin{equation}
  \mathcal U_{q+1}=\mathcal U_q+W_{q+1},
\end{equation}
where $W_{q+1}$ is the sum of the new kernels containing
$i_{q+1}$.  By canonicality in the coordinate $i_{q+1}$,
\begin{equation}\label{eq:Wcenter}
  \E(W_{q+1}\mid\cF_q,X_1)=0.
\end{equation}
Using the tower property and \eqref{eq:Wcenter},
\begin{align*}
  \E(Z_{q+1}\mid\cF_q)
  &=n\E(\mathcal U_{q+1}^2\mid\cF_q)\\
  &=n\E\{(\mathcal U_q+W_{q+1})^2\mid\cF_q\}\\
  &=Z_q+n\E(W_{q+1}^2\mid\cF_q)\\
  &\ge Z_q.
\end{align*}
Thus $Z_q$ is a nonnegative submartingale.  For $q=n-1$, orthogonality gives
\begin{align*}
  \E Z_{n-1}
  &=n\E \mathcal U_{n-1}^2\\
  &=n\sum_{r=1}^k\binom{n-1}{r-1}\E T_{1\cdots r}^2\\
  &=\sum_{r=1}^k r\binom nr\E T_{1\cdots r}^2
  \le k\gamma.
\end{align*}
Finally, \eqref{eq:PB} follows from Doob's inequality.
\end{proof}

Because $q(t)$ is nondecreasing in $|t|$, the events $B_{q(|t|)}$ form an
increasing family as $|t|$ increases.

\subsection{The exceptional event}

The monotonicity of the exceptional events allows us to use the signed
integral estimate \eqref{eq:Lambda-osc} before taking expectations.

\begin{lemma}\label{lem:nested-bad}
Let $\{B(r):t_*\le r<a\}$ be a nondecreasing family of events whose
indicators are jointly measurable in $(r,\omega)$.  Then, for every
real random variable $Y$,
\begin{equation}
  \sup_x\left|
    \int_{t_*<|t|<a}e^{-\ii tx}\Lambda_{\sigma,a}(t)
      \E\{\1_{B(|t|)}e^{\ii tY}\}\,\dd t
  \right|
  \le C\Pp(B(a-)),
\end{equation}
where $B(a-)=\bigcup_{t_*\le r<a}B(r)$.  By monotonicity, this union may be
taken over rational $r$ in $[t_*,a)$, so it is measurable.
\end{lemma}

\begin{proof}
For each sample point $\omega$, consider the positive
$t\in[t_*,a)$ for which $B(t)$ occurs.  By monotonicity, this set is either empty
or, up to its left endpoint, an interval $[\tau(\omega),a)$.  Therefore \eqref{eq:Lambda-osc}, with
$z=Y(\omega)-x$, gives the pathwise estimate
\[
  \left|\int_{t_*<|t|<a}e^{\ii t(Y(\omega)-x)}
     \Lambda_{\sigma,a}(t)\1_{B(|t|)}(\omega)\,\dd t\right|
  \le C\1_{B(a-)}(\omega).
\]
Since $t_*>0$ and $a<\infty$, the integral of
$|\Lambda_{\sigma,a}|$ over this annulus is finite.  Taking expectations
and applying Fubini's theorem proves the result.
\end{proof}

Since $q(r)$ is a finite-valued step function, the indicators of
$B(r)=B_{q(r)}$ are jointly measurable.  The event $B(a-)$ equals $B_{q_*}$ with
$q_*=\sup_{t_*\le r<a}q(r)$, because the indices are integers and the
events $B_q$ are nested.  Hence \eqref{eq:PB} applies directly to
$B(a-)$.

The contribution from $B_q$ to \eqref{eq:R-exact} is
\begin{align}
 \E\{\1_{B_q}e^{\ii t(L+T^{(0)})}\Psi_t(U+V)\}
 &=\E(\1_{B_q}e^{\ii tS})
   -\E\{\1_{B_q}e^{\ii t(L+T^{(0)})}\}\notag\\
 &\quad-\ii t\,\E\{\1_{B_q}e^{\ii t(L+T^{(0)})}(U+V)\}.
\end{align}
Apply \Cref{lem:nested-bad} with $B(r)=B_{q(r)}$ and $Y=S$.  By
\eqref{eq:PB}, with a fixed $\eta=\eta_k>0$,
\begin{equation}\label{eq:bad-S}
  \sup_x\left|
    \int_{t_*<|t|<a}e^{-\ii tx}\Lambda_{\sigma,a}(t)
      \E\{\1_{B_q}e^{\ii tS}\}\,\dd t
  \right|
  \le C_k\gamma.
\end{equation}

The event $B_q$ is $\cF_q$-measurable, and $T^{(0)}$ depends only on the
variables outside $M$.  Conditioning on the outside variables and integrating
over $X_1,\dots,X_m$ gives
\begin{equation}
  \E\{\1_{B_q}e^{\ii t(L+T^{(0)})}\mid\cF_q\}
  =\1_{B_q}e^{\ii t(L_{M^c}+T^{(0)})}\vartheta(t)^m.
\end{equation}
Hence
\begin{equation}
  \left|\E\1_{B_q}e^{\ii t(L+T^{(0)})}\right|
  \le C_k\gamma|\vartheta(t)|^m.
\end{equation}
Using \eqref{eq:Lambda-size} and \eqref{eq:damping-polynomial},
\begin{equation}\label{eq:bad-f0-int}
  \int_{t_*<|t|<a}|\Lambda_{\sigma,a}(t)|
  \left|\E\1_{B_q}e^{\ii t(L+T^{(0)})}\right|\dd t
  \le C_k\gamma.
\end{equation}

For the $V$-part of $f_1$, the Cauchy--Schwarz inequality, \eqref{eq:UV-variance}, and
\eqref{eq:PB} yield
\begin{equation}
  \E(\1_{B_q}|V|)
  \le \Pp(B_q)^{1/2}(\E V^2)^{1/2}
  \le C_k\gamma\frac mn.
\end{equation}
Since $|t\Lambda_{\sigma,a}(t)|\le C$ and
$m/n\le A_k\log|t|/t^2$,
\begin{equation}\label{eq:bad-V-int}
  \int_{t_*<|t|<a}|t\Lambda_{\sigma,a}(t)|
  \left|\E\1_{B_q}e^{\ii t(L+T^{(0)})}V\right|\dd t
  \le C_k\gamma.
\end{equation}

For the $U$-part, condition on the outside variables and write
\begin{equation}
  U=\sum_{j=1}^mU_j,
  \qquad
  U_j=\sum_{r=1}^k
      \sum_{\substack{B\subset M^c\\|B|=r-1}}
      T_{\{j\}\cup B}.
\end{equation}
Conditionally on $\cF_q$, the variables $U_1,\dots,U_m$ depend on distinct
observations in the block and are therefore independent.  Canonicality gives
\begin{equation}\label{eq:Uj-centered}
  \E(U_j\mid\cF_q)=0.
\end{equation}
Moreover, once the outside variables are fixed, the function of $X_j$
defining $U_j$ is the same for every $j\in M$.  Hence their conditional
variances are equal.  Write
\begin{equation}
  v_q=\E(U_j^2\mid\cF_q),\qquad j\in M.
\end{equation}
For $j=1$, $U_j$ is the variable $\mathcal U_q$ in \eqref{eq:Uq}.  Hence
$Z_q=nv_q$.  Orthogonality also gives the explicit identity
\[
  \E v_q=\sum_{r=1}^k\binom q{r-1}\E T_{1\cdots r}^2
  \le \frac1n\sum_{r=1}^k r\gamma_r.
\]
Thus
\begin{equation}\label{eq:Evq}
  \E v_q\le k\gamma/n.
\end{equation}

Set
\begin{equation}
  a_j=\E(e^{\ii tL_j}U_j\mid\cF_q).
\end{equation}
Using \eqref{eq:Uj-centered},
\begin{align}
  |a_j|
  &=\left|\E\{(e^{\ii tL_j}-1)U_j\mid\cF_q\}\right|\notag\\
  &\le \{\E|e^{\ii tL_1}-1|^2\}^{1/2}v_q^{1/2}
  \le |t|n^{-1/2}v_q^{1/2}.
  \label{eq:aj-bound}
\end{align}
Conditional independence gives
\begin{equation}
  \E(e^{\ii tL_M}U\mid\cF_q)
  =\vartheta(t)^{m-1}\sum_{j=1}^ma_j.
\end{equation}
Consequently,
\begin{align}
  \left|\E\1_{B_q}e^{\ii t(L+T^{(0)})}U\right|
  &\le |\vartheta(t)|^{m-1}m|t|n^{-1/2}
      \E(\1_{B_q}v_q^{1/2})\notag\\
  &\le |\vartheta(t)|^{m-1}m|t|n^{-1/2}
      (\E v_q)^{1/2}\Pp(B_q)^{1/2}\notag\\
  &\le C_k\gamma|\vartheta(t)|^{m-1}\frac{m|t|}{n}.
\end{align}
Since $m|t|/n\le A_k\log|t|/|t|$ and
$|\vartheta|^{m-1}\le|\vartheta|^{m/2}\le |t|^{-A_k/16}$, we obtain
\begin{equation}\label{eq:bad-U-int}
  \int_{t_*<|t|<a}|t\Lambda_{\sigma,a}(t)|
  \left|\E\1_{B_q}e^{\ii t(L+T^{(0)})}U\right|\dd t
  \le C_k\gamma.
\end{equation}
Equations \eqref{eq:bad-S}, \eqref{eq:bad-f0-int},
\eqref{eq:bad-V-int}, and \eqref{eq:bad-U-int} bound the contribution
from $B_q$ to \eqref{eq:integrated-P1}.

\subsection{The product estimate on the complementary event}

On $H_q$,
\begin{equation}
 v_q\le\eta/n.
\end{equation}
Conditionally on $\cF_q$, the pairs $(L_j,U_j)$ are independent and have
a common distribution.  Define
\begin{equation}
 b=\ii t\,\E(e^{\ii tL_1}U_1\mid\cF_q),\qquad
 r=\E\{e^{\ii tL_1}(e^{\ii tU_1}-1-\ii tU_1)\mid\cF_q\}.
\end{equation}
The conditional Taylor remainder is therefore
\begin{align}
 Q_m(t)&=\E\{e^{\ii tL_M}(e^{\ii tU}-1-\ii tU)\mid\cF_q\}\notag\\
       &=(\vartheta+b+r)^m-\vartheta^m-m\vartheta^{m-1}b.
\end{align}
For complex $z,d$ and $m\ge2$, Taylor's formula along the line segment
from $z$ to $z+d$ gives
\[
 |(z+d)^m-z^m-mz^{m-1}d|
 \le\tfrac12m(m-1)(|z|+|d|)^{m-2}|d|^2.
\]
Consequently,
\begin{equation}\label{eq:product-bound}
 |Q_m(t)|\le m|\vartheta|^{m-1}|r|
       +\tfrac12m(m-1)(|\vartheta|+|b|+|r|)^{m-2}(|b|+|r|)^2.
\end{equation}
By \eqref{eq:aj-bound} and $|e^{\ii x}-1-\ii x|\le x^2/2$,
\begin{equation}
 |b|\le t^2\sqrt{v_q/n},\qquad |r|\le t^2v_q/2.
\end{equation}
Thus, on $H_q$,
\[
 |b|+|r|\le(\sqrt\eta+\eta/2)t^2/n.
\]
Since $t^2/n\le1$ and
\[
 \inf_{0<x\le1}\frac{e^{-x/8}-e^{-x/4}}x>0,
\]
we can choose $\eta>0$ so small that Lemma~\ref{lem:theta} implies
\begin{equation}\label{eq:sum-damping}
 |\vartheta|+|b|+|r|\le e^{-t^2/(8n)}\qquad\hbox{on }H_q.
\end{equation}
Combining \eqref{eq:product-bound}--\eqref{eq:sum-damping} and using $m\ge8$,
we obtain
\[
 \1_{H_q}|Q_m(t)|
 \le C e^{-t^2m/(16n)}\1_{H_q}
       \{mt^2v_q+m^2t^4(v_q/n+v_q^2)\}.
\]
On $H_q$, $v_q^2\le\eta v_q/n$.  Taking expectations and using
\eqref{eq:Evq} gives
\begin{equation}
 \E\{\1_{H_q}|Q_m(t)|\}
 \le C_k\gamma e^{-t^2m/(16n)}
       \left\{\frac{t^2m}{n}+\left(\frac{t^2m}{n}\right)^2\right\}
 \le C_k\gamma(\log|t|)^2|t|^{-A_k/32}.
\end{equation}
It follows from \eqref{eq:Lambda-size} that
\begin{equation}\label{eq:Q-integrated}
 \int_{t_*<|t|<a}|\Lambda_{\sigma,a}(t)|
                  \E\{\1_{H_q}|Q_m(t)|\}\,\dd t\le C_k\gamma.
\end{equation}

It remains to estimate the terms containing $V$.  For real $u,v$,
\begin{align}
 |\Psi_t(u+v)-\Psi_t(u)|
 &=|\ii tv(e^{\ii tu}-1)+e^{\ii tu}(e^{\ii tv}-1-\ii tv)|\notag\\
 &\le t^2|uv|+t^2v^2/2.
 \label{eq:V-deterministic}
\end{align}
By \eqref{eq:UV-variance},
\begin{equation}
 \E|UV|\le(\E U^2\,\E V^2)^{1/2}
          \le C_k\gamma(m/n)^{3/2}.
\end{equation}
Hence \eqref{eq:Lambda-size} and \eqref{eq:mt-two-sided} yield
\begin{align}
 \int_{t_*<|t|<a}|\Lambda_{\sigma,a}(t)|t^2\E|UV|\,\dd t
 &\le C_k\gamma\int_{t_*}^\infty\frac{(\log t)^{3/2}}{t^2}\,\dd t
 \le C_k\gamma,\\
 \int_{t_*<|t|<a}|\Lambda_{\sigma,a}(t)|t^2\E V^2\,\dd t
 &\le C_k\gamma\int_{t_*}^\infty\frac{(\log t)^2}{t^3}\,\dd t
 \le C_k\gamma.\label{eq:V2-integrated}
\end{align}
Finally,
\[
 \E\{\1_{H_q}e^{\ii t(L+T^{(0)})}\Psi_t(U)\}
 =\E\{\1_{H_q}e^{\ii t(L_{M^c}+T^{(0)})}Q_m(t)\}.
\]
Its absolute value is bounded by $\E\{\1_{H_q}|Q_m(t)|\}$.
Together with \eqref{eq:V-deterministic}, equations
\eqref{eq:Q-integrated}--\eqref{eq:V2-integrated} bound the contribution
from $H_q$.  The estimates for $B_q$ then prove \eqref{eq:integrated-P1}.

We now combine the estimates for the signed Fourier integral.  The low-frequency
part is bounded by \eqref{eq:low-integral}.  In the high-frequency region,
\eqref{eq:three-way}, \eqref{eq:integrated-P1},
\eqref{eq:prop2-integrated}, and \eqref{eq:prop3-integrated} give
\[
  \sup_x\left|
    \int_{t_*<|t|<a}e^{-\ii tx}\Lambda_{\sigma,a}(t)
       \{f(t)-g_k(t)\}\,\dd t
  \right|
  \le C_k(\beta+\gamma).
\]
Thus \eqref{eq:signed-target} holds.  Applying \Cref{lem:smoothing} and using
$a^{-1}\le C_k(\beta+\gamma)$ yields
\[
  \norm{F_S-N_k}_\infty
  \le C_ka^{-1}+C_k(\beta+\gamma)
  \le C_k(\beta+\gamma).
\]
This proves \Cref{thm:main}.

\subsection{Optimality}
\label{sec:main-sharp}

We next show that both terms in \eqref{eq:main} are necessary.  For a
sequence $S_n=L_n+T_n$, let $\beta_n$, $\gamma_n$, and $N_{k,n}$ denote the
quantities defined in Section~\ref{sec:main-results}.

\begin{proposition}\label{prop:main-sharp}
For every fixed $k\ge2$, there are examples satisfying the assumptions
of Theorem~\ref{thm:main}.
\begin{enumerate}[label=\textup{(\roman*)}]
\item There is a sequence with $T_n=0$, $\beta_n\to0$, and
\[
 \norm{F_{S_n}-N_{k,n}}_\infty\asymp\beta_n.
\]
\item There is a sequence with $\gamma_n\to0$, $\beta_n=o(\gamma_n)$, and
\[
 \norm{F_{S_n}-N_{k,n}}_\infty\asymp\gamma_n.
\]
In this example, $T_n$ is canonical of order two and all adjustment
coefficients vanish.
\end{enumerate}
\end{proposition}

\begin{proof}
For the first assertion, take $L_n=n^{-1/2}\sum_i\varepsilon_i$, where the
$\varepsilon_i$ are independent and take the values $-1,1$ with equal
probabilities.  Then $\beta_n=n^{-1/2}$ and $N_{k,n}=\Phi$.  A distribution
with an atom of mass $a$ has Kolmogorov distance at least $a/2$ from every continuous
distribution function.  Stirling's formula gives
\[
 \max_x\Pp(L_n=x)=2^{-n}\binom n{\lfloor n/2\rfloor}
       \sim\sqrt{\frac{2}{\pi n}}.
\]
This proves the lower bound.  Theorem~\ref{thm:main} gives the upper bound.

For the second assertion, let $X_i=(Z_i,J_i)$, where $Z_i$ is standard
normal and, independently, $J_i$ is uniform on $\{1,\ldots,m_n\}$, with
$m_n\ge2$.  Define
\begin{equation}\label{eq:collision-definitions}
 \begin{split}
 L_n&=n^{-1/2}\sum_iZ_i,\\
 D_n&=\sum_{i<j}\1_{\{J_i=J_j\}},\qquad
 \lambda_n=\binom n2/m_n,\qquad T_n=D_n-\lambda_n.
 \end{split}
\end{equation}
The kernel $\1_{\{J_i=J_j\}}-m_n^{-1}$ is canonical.  The collision
indicators are pairwise independent, even when the corresponding pairs
share one index.  Consequently,
\begin{equation}\label{eq:collision-moments}
 \begin{split}
 \E D_n&=\lambda_n,\qquad
 \gamma_n=\Var D_n=\lambda_n(1-m_n^{-1}),\\
 \E\{D_n(D_n-1)\}
 &=\binom n2\left(\binom n2-1\right)m_n^{-2}\le\lambda_n^2.
 \end{split}
\end{equation}
Moreover, $L_n$ is standard normal and independent of $T_n$.  All
adjustment coefficients vanish, and
$\beta_n=2\sqrt{2/\pi}\,n^{-1/2}$.

Suppose $\lambda_n\to0$.  The last bound in
\eqref{eq:collision-moments} implies
\[
 \Pp(D_n=0)=1-\lambda_n+O(\lambda_n^2),\qquad
 \Pp(D_n=1)=\lambda_n+O(\lambda_n^2),\qquad
 \Pp(D_n\ge2)=O(\lambda_n^2).
\]
Indeed, both $\Pp(D_n\ge2)$ and
$\E\{D_n\1_{\{D_n\ge2\}}\}$ are bounded by
$\E\{D_n(D_n-1)\}$.  Conditioning on $D_n$ and expanding $\Phi$ at
$0$ and $-1$ now give
\begin{equation}\label{eq:rare-shift}
 \begin{split}
 F_{L_n+T_n}(0)-\tfrac12
 &=\E\Phi(\lambda_n-D_n)-\tfrac12\\
 &=c_*\lambda_n+O(\lambda_n^2),\qquad
 c_*=\phi(0)+\Phi(-1)-\tfrac12>0.
 \end{split}
\end{equation}
Here $\phi=\Phi'$ and
$c_*=\int_0^1\{\phi(0)-\phi(x)\}\,\dd x$.
This calculation uses only that $D_n$ is a nonnegative integer-valued random
variable with mean $\lambda_n$ and second factorial moment $O(\lambda_n^2)$,
independent of the standard normal variable $L_n$.

Choose $m_n=\lceil n^{9/4}\rceil$.  Then
$\gamma_n\sim\lambda_n\sim\tfrac12n^{-1/4}$ and
$\beta_n=o(\gamma_n)$.  Equation~\eqref{eq:rare-shift} gives the required
lower bound, while Theorem~\ref{thm:main} gives the upper bound.
\end{proof}

The collision calculation \eqref{eq:rare-shift} will also give the
lower bound for Studentization.  At threshold zero, division by a
positive jackknife standard deviation does not change the event.

\section{Studentization at \texorpdfstring{$5/3$}{5/3}}
\label{sec:studentized}

We retain the notation of Theorem~\ref{thm:studentized}; in particular,
$g$ is the first projection of $h$.  We compare the Studentized statistic
with two polynomials and apply Corollary~\ref{cor:finite-order}.
The truncation probability, the remainder variance, and the absolute
value of the signed expectation under a bounded change of measure are
all bounded at the same order.

\subsection{Reduction and truncation}

The following elementary lemma compares the ratio with two polynomials.

\begin{lemma}\label{lem:student-sandwich}
Let $W$ and $\Delta$ be real random variables with $1+\Delta\ge0$.  Define
$Z=W/\sqrt{1+\Delta}$ when $1+\Delta>0$, and assign any value to $Z$
otherwise.  Let
\[
 Y_r=W(1-\Delta/2+r\Delta^2),\qquad r=0,1.
\]
Then
\begin{equation}\label{eq:student-sandwich}
 \norm{F_Z-\Phi}_\infty
 \le\max_{r=0,1}\norm{F_{Y_r}-\Phi}_\infty
       +\Pp(|\Delta|>1/2).
\end{equation}
\end{lemma}

\begin{proof}
For $|x|\le1/2$,
\[
 1-x/2\le(1+x)^{-1/2}\le1-x/2+x^2.
\]
All three factors are positive.  On $\{|\Delta|\le1/2,W\ge0\}$,
$Y_0\le Z\le Y_1$, whereas the inequalities are reversed when $W<0$.
If $z\ge0$, all three variables are below $z$ on
$\{|\Delta|\le1/2,W<0\}$.  Hence
\[
 \{Y_1\le z\}\subseteq\{Z\le z\}\subseteq\{Y_0\le z\}
 \quad\text{on }\{|\Delta|\le1/2\},\qquad z\ge0.
\]
For $z<0$, the roles of $Y_0$ and $Y_1$ are interchanged.  Allowing for the
exceptional event proves \eqref{eq:student-sandwich}.
\end{proof}

Replacing $h$ by $(h-\theta)/\sigma_g$, we may assume $\theta=0$ and
$\sigma_g=1$.  Let
\begin{equation}
 a=\E|g_1|^3,\qquad b=\E|\eta_{12}|^{5/3},\qquad
 \rho=(a+b)/\sqrt n.
\end{equation}
We may assume that these moments are finite.  Since $a\ge1$, it suffices
to prove the theorem under the condition $\rho\le c_0$, where $c_0>0$ is
a sufficiently small absolute constant.  In particular,
$n\ge10$ for sufficiently small $c_0$.  Truncation at $n^{3/2}$ gives the
following tail and second-moment bounds:
\begin{equation}
 n^2\Pp(|\eta_{12}|>n^{3/2})\le\frac b{\sqrt n},\qquad
 \frac1n\E\{\eta_{12}^2\1_{\{|\eta_{12}|\le n^{3/2}\}}\}
       \le\frac b{\sqrt n}.
\end{equation}
We also truncate according to a conditional moment to obtain the
uniform row bounds needed for the jackknife denominator.

Define
\begin{equation}
 m(x)=\int|\eta(x,y)|^{5/3}\,P(\dd y),\qquad
 D=\{x:m(x)>n^{3/2}\},\qquad \delta=P(D),
\end{equation}
and write $J(x)=\1_{D^c}(x)$.  Let
\begin{align}
 g^\flat(x)&=g(x)\1_{\{|g(x)|\le\sqrt n\}},\notag\\
 \eta^\flat(x,y)&=\eta(x,y)J(x)J(y)
                    \1_{\{|\eta(x,y)|\le n^{3/2}\}}.
\end{align}
Define the truncation event by
\[
 \mathcal E_n=
 \bigcup_{i=1}^n\{|g_i|>\sqrt n\}
 \;\cup\;\bigcup_{i=1}^n\{m(X_i)>n^{3/2}\}
 \;\cup\;\bigcup_{1\le i<j\le n}\{|\eta_{ij}|>n^{3/2}\}.
\]
On $\mathcal E_n^c$, the truncated and original kernel values agree for
all observed pairs, and
\begin{align}
 \Pp(\mathcal E_n)
 &\le n\Pp(|g_1|>\sqrt n)+n\Pp(m(X_1)>n^{3/2})
       +\binom n2\Pp(|\eta_{12}|>n^{3/2})\notag\\
 &\le C\rho.
 \label{eq:student-censor-prob}
\end{align}
Here $\delta\le b n^{-3/2}$.

To center the truncated kernel, let
\begin{equation}\label{eq:student-centering-defs}
 m_0=\E g^\flat(X_1),\qquad
 r(x)=\int\eta^\flat(x,y)\,P(\dd y),\qquad c=\E r(X_1),
\end{equation}
and define
\begin{align}
 p(x)&=r(x)-c,\notag\\
 \eta_0(x,y)&=\eta^\flat(x,y)-r(x)-r(y)+c,\notag\\
 g_*(x)&=g^\flat(x)-m_0+p(x),\qquad s=\{\E g_*(X_1)^2\}^{1/2}.
 \label{eq:student-censored-hoeffding}
\end{align}
Thus $g_*$ is centered, $\eta_0$ is canonical, and the truncated kernel is
\begin{equation}\label{eq:student-hflat}
 h^\flat(x,y)=g^\flat(x)+g^\flat(y)+\eta^\flat(x,y)
 =2m_0+c+g_*(x)+g_*(y)+\eta_0(x,y).
\end{equation}
Essential suprema below are taken with respect to $P$ or $P\otimes P$.
We also use $C_0$ for an absolute constant.  Subscripts on constants
indicate dependence on the corresponding parameters.  The next lemma bounds
the centering terms and the change in the variance of the first projection.

\begin{lemma}
The quantities in \eqref{eq:student-centering-defs}--\eqref{eq:student-censored-hoeffding}
satisfy
\begin{align}
 |m_0|&\le a/n,& |c|&\le Cb/n,\notag\\
 \norm{p}_\infty&\le C\sqrt n,&
 \E p(X_1)^2&\le Cb/\sqrt n,&
 \E|p(X_1)|^3&\le Cb,
 \label{eq:student-centering-bounds}\\
 |s^2-1|&\le C\rho,&
 \E|g_*(X_1)|^3&\le C(a+b),&
 \norm{g_*}_\infty&\le C\sqrt n.
 \label{eq:student-first-projection}
\end{align}
Moreover,
\begin{align}
 \E\eta_0(X_1,X_2)^2&\le\sqrt n\,b,
 &\E|\eta_0(X_1,X_2)|^{5/3}&\le Cb,\notag\\
 \mathop{\rm ess\,sup}_{x}\int\eta_0(x,y)^2\,P(\dd y)&\le Cn^2,
 &\norm{\eta_0}_\infty&\le Cn^{3/2}.
 \label{eq:student-eta0-bounds}
\end{align}
\end{lemma}

\begin{proof}
The bound for $m_0$ follows from $\E g=0$ and
$\E|g|\1_{\{|g|>\sqrt n\}}\le a/n$.  Canonicality gives
\[
 \E\{\eta_{12}J(X_1)J(X_2)\}
 =\E\{\eta_{12}\1_D(X_1)\1_D(X_2)\}.
\]
H\"older's inequality and independence therefore imply
\begin{align}
 |c|&\le b/n+b^{3/5}\delta^{4/5}\notag\\
 &\le\frac bn\{1+(b/\sqrt n)^{2/5}\}\le 2b/n.
 \label{eq:student-c-small}
\end{align}

Canonicality also gives
\begin{align}
 r(x)={}&-J(x)\int_D\eta(x,y)\,P(\dd y)\notag\\
 &-J(x)\int\eta(x,y)J(y)
                  \1_{\{|\eta(x,y)|>n^{3/2}\}}\,P(\dd y).
 \label{eq:student-r-decomposition}
\end{align}
The absolute values of the two terms are bounded, respectively, by
\[
 R_D(x)=J(x)m(x)^{3/5}\delta^{2/5},\qquad
 R_T(x)=J(x)m(x)/n.
\]
Since $m\le n^{3/2}$ on $D^c$ and $\E m(X_1)=b$,
\begin{align}
 \norm{R_T}_\infty&\le\sqrt n,&
 \E R_T(X_1)^2&\le b/\sqrt n,&
 \E R_T(X_1)^3&\le b,\notag\\
 \norm{R_D}_\infty&\le\sqrt n\,(b/\sqrt n)^{2/5},&
 \E R_D(X_1)^2&\le(b/\sqrt n)^{9/5},&
 \E R_D(X_1)^3&\le b(b/\sqrt n)^{6/5}.
\end{align}
For example,
$\E R_D^2\le n^{3/10}b\delta^{4/5}$ and
$\E R_D^3\le n^{6/5}b\delta^{6/5}$.
Since $b/\sqrt n\le c_0$, these estimates and \eqref{eq:student-c-small}
prove the bounds for $p$ in \eqref{eq:student-centering-bounds}.

To prove \eqref{eq:student-first-projection}, we bound the covariance
between $g^\flat(X_1)$ and $p(X_1)$.
The truncation of $g$ changes its variance by at most $2a/\sqrt n$.
The term containing $R_T$ in \eqref{eq:student-r-decomposition} satisfies
\[
 \E\{|g^\flat(X_1)|R_T(X_1)\}\le b/\sqrt n.
\]
For the other term, H\"older's inequality with exponents $5/3$ and $5/2$
gives
\begin{align}
 &\E\{|g^\flat(X_1)\eta_{12}|J(X_1)\1_D(X_2)\}\notag\\
 &\qquad\le b^{3/5}
       \{\E|g^\flat(X_1)|^{5/2}\,\delta\}^{2/5}\notag\\
 &\qquad\le a^{1/5}b^{3/5}\delta^{2/5}
       \le\frac b{\sqrt n}(a/\sqrt n)^{1/5}
       \le Cb/\sqrt n.
\end{align}
Here we used
$\E|g|^{5/2}\le(\E g^2)^{1/2}(\E|g|^3)^{1/2}=a^{1/2}$.
Also, $|m_0c|\le Cab/n^2\le Cb/\sqrt n$.  Hence
$|\Cov(g^\flat(X_1),p(X_1))|\le Cb/\sqrt n$.  The identity
\[
 s^2=\Var(g^\flat(X_1))
       +2\Cov(g^\flat(X_1),p(X_1))+\E p(X_1)^2
\]
proves $|s^2-1|\le C\rho$.  The third-moment and essential supremum bounds follow
from \eqref{eq:student-centering-bounds} and $|g^\flat|\le\sqrt n$.

Finally,
\[
 \E(\eta^\flat)^2\le n^{1/2}b,\qquad
 \int\eta^\flat(x,y)^2\,P(\dd y)
       \le n^{1/2}J(x)m(x)\le n^2.
\]
The canonical projection is an orthogonal projection in $L^2(P^2)$,
which proves the first bound in \eqref{eq:student-eta0-bounds}.
Each of the four terms defining $\eta_0$ has $L^{5/3}$ norm at most
that of $\eta^\flat$.  Thus
$\E|\eta_0|^{5/3}\le Cb$.  The conditional second-moment bound follows
from the preceding row bound and
$\norm r_\infty\le C\sqrt n$, $\E r^2\le Cb/\sqrt n$.
The pointwise bound follows from $|\eta^\flat|\le n^{3/2}$ and
\eqref{eq:student-censored-hoeffding}.
\end{proof}

Choose $c_0$ so that $1/2\le s^2\le2$.  Since $s>0$, we may define
\begin{equation}\label{eq:student-uv}
 u(x)=\frac{g_*(x)}{s\sqrt n},\qquad
 v(x,y)=\frac{\eta_0(x,y)}{s\sqrt n(n-1)},\qquad
 d_n=\frac{\sqrt n\,(m_0+c/2)}s.
\end{equation}
Write $u_i=u(X_i)$ and $v_{ij}=v(X_i,X_j)$.  The preceding lemma gives
\begin{equation}\label{eq:student-u-basic}
 \E u_1=0,\quad n\E u_1^2=1,\quad
 \norm u_\infty\le C,\quad n\E|u_1|^3\le C\rho,
 \quad |d_n|\le C\rho,
\end{equation}
and
\begin{equation}
 \E(v_{12}\mid X_1)=\E(v_{12}\mid X_2)=0,
 \qquad \norm v_\infty\le C,
\end{equation}
\begin{equation}\label{eq:student-v-L2}
 n^2\E v_{12}^2\le C\rho,\qquad
 n\mathop{\rm ess\,sup}_{x}\int v(x,y)^2\,P(\dd y)\le C.
\end{equation}
The mixed moment satisfies
\begin{equation}\label{eq:student-one-edge}
 n^2\E|u_1u_2v_{12}|\le C\rho.
\end{equation}
Indeed, H\"older's inequality and
$\E|g_*|^{5/2}\le(\E g_*^2)^{1/2}(\E|g_*|^3)^{1/2}$ give
\begin{align*}
 \E|g_*(X_1)g_*(X_2)\eta_0(X_1,X_2)|
 &\le \{\E|g_*|^{5/2}\}^{4/5}\{\E|\eta_0|^{5/3}\}^{3/5}\\
 &\le C(a+b)^{2/5}b^{3/5}\le C(a+b).
\end{align*}
The normalization in \eqref{eq:student-uv} proves
\eqref{eq:student-one-edge}.

\subsection{The denominator polynomial}

Let $U_n^\flat$ and $\widehat\sigma_\flat^2$ denote the statistic and
jackknife estimator formed from $h^\flat$.  Let
\[
 U=\sum_i u_i,\qquad V=\sum_{i<j}v_{ij},\qquad
 W=d_n+U+V,\qquad r_i=\sum_{j\ne i}v_{ij}.
\]
By \eqref{eq:student-hflat},
\begin{equation}\label{eq:student-exact-ratio}
 \frac{\sqrt n\,U_n^\flat}{2\widehat\sigma_\flat}
 =\frac{W}{\sqrt{1+\Delta}},\qquad
 1+\Delta=\frac{\widehat\sigma_\flat^2}{s^2}.
\end{equation}
We use the same convention on both sides when the denominator is zero.
In particular,
$1+\Delta\ge0$.

Substitution of \eqref{eq:student-hflat} into the jackknife formula gives
\begin{align}
 1+\Delta={}&\frac n{n-1}\sum_i(u_i-U/n)^2\notag\\
 &+\frac{2n}{n-2}\sum_i(u_i-U/n)(r_i-2V/n)\notag\\
 &+\frac{n(n-1)}{(n-2)^2}\sum_i(r_i-2V/n)^2.
\end{align}
Consequently,
\begin{equation}
 \Delta=\Delta_1+\Delta_2+\Delta_3,
\end{equation}
where
\begin{align}
 \Delta_1&=\sum_i(u_i^2-1/n)-\frac2{n-1}\sum_{i<j}u_i u_j,
 \\
 \Delta_2&=\frac{2n}{n-2}
       \left\{\sum_{i<j}(u_i+u_j)v_{ij}-\frac2nUV\right\},
 \label{eq:student-B}\\
 \Delta_3&=\frac{n(n-1)}{(n-2)^2}
       \left\{\sum_i r_i^2-\frac4nV^2\right\}.
 \label{eq:student-C}
\end{align}
Here $\Delta_1$ is a sum of centered squares and an off-diagonal quadratic term.
Each term of $\Delta_2$ contains a canonical pair factor, and each term of $\Delta_3$
contains at least two such factors.  All deterministic coefficients are
bounded uniformly for $n\ge4$.  We retain the factor $(n-1)^{-1}$ in the
second term of $\Delta_1$ when estimating moments.

For $r=0,1$, define
\begin{equation}\label{eq:student-Yr}
 Y_r=W(1-\Delta/2+r\Delta^2),\qquad R_r=Y_r-U.
\end{equation}
The statistic $Y_r$ is symmetric and has Hoeffding order at most ten,
since $W$ has order at most two and $\Delta$ has order at most four.
We next estimate the variance of $R_r$ and its expectation under the
laws $P_t$ defined below.

\subsection{Moment estimates}

We use the scalar and Hilbert space forms of Rosenthal's inequality.
Let $\xi_i$ be independent centered random vectors in a finite-dimensional
Hilbert space with norm $\|\cdot\|$.  For $p\ge2$,
\begin{equation}\label{eq:rosenthal-H}
 \E\Big\|\sum_i\xi_i\Big\|^p
 \le C_p\left\{\Big(\sum_i\E\|\xi_i\|^2\Big)^{p/2}
                   +\sum_i\E\|\xi_i\|^p\right\}.
\end{equation}
See \cite[Theorem~5.3]{Pinelis1994}; the constant $C_p$ is independent
of the dimension.
For independent nonnegative random variables $\zeta_i$ and $p\ge1$, we
also use
\begin{equation}\label{eq:rosenthal-positive}
 \E\Big(\sum_i\zeta_i\Big)^p
 \le C_p\left\{\Big(\sum_i\E\zeta_i\Big)^p
                   +\sum_i\E\zeta_i^p\right\};
\end{equation}
see \cite[Section~2.1]{GineLatalaZinn2000}.  The constants depend only
on $p$.  The following lemma gives moment bounds for a canonical pair sum and the
vector of its row sums.

\begin{lemma}\label{lem:student-kernel-mom}
Let $X_1,\ldots,X_n$ be independent with common law $P$, and let
$a_0(x,y)$ be a symmetric canonical kernel.  Suppose $0<\rho\le1$ and
that, for some $K\ge1$,
\[
 \norm{a_0}_\infty\le K,\qquad
 n\mathop{\rm ess\,sup}_x\E a_0(x,X_1)^2\le K,\qquad
 n^2\E a_0(X_1,X_2)^2\le K\rho.
\]
For every fixed $p\ge2$,
\begin{equation}\label{eq:student-kernel-mom}
 \E\left|\sum_{i<j}a_0(X_i,X_j)\right|^p
 +\E\left\{\sum_i\left(\sum_{j\ne i}a_0(X_i,X_j)\right)^2\right\}^{p/2}
 \le C_{p,K}\rho.
\end{equation}
\end{lemma}

\begin{proof}
First consider two independent groups $(X_i)_{i\in I}$ and
$(Y_j)_{j\in J}$, all with law $P$.  If either group is empty, both sums vanish.  By
relabelling, assume $I=\{1,\ldots,r\}$ and $J=\{1,\ldots,s\}$, where
$1\le r,s\le n$.  Write $Y=(Y_j)_{j\in J}$ and set
\[
 R_i=\sum_{j\in J}a_0(X_i,Y_j),\qquad
 b_0(x)=\E a_0(x,X_1)^2.
\]
Conditionally on the first group, the vectors
$(a_0(X_i,Y_j))_{i\in I}$ are independent and centered.  By
\eqref{eq:rosenthal-H},
\begin{align*}
 \E\Big(\sum_{i\in I}R_i^2\Big)^{p/2}
 &\le C_p\E\Big(n\sum_{i\in I}b_0(X_i)\Big)^{p/2}\\
 &\quad+C_p n\E\Big(\sum_{i\in I}a_0(X_i,Y_1)^2\Big)^{p/2}.
\end{align*}
Apply \eqref{eq:rosenthal-positive} to the first sum and, conditionally on
$Y_1$, to the second.  Symmetry of $a_0$ gives the upper bound
\begin{equation}
 C_p\left\{(n^2\E a_{0,12}^2)^{p/2}
      +n^{1+p/2}\E b_0(X_1)^{p/2}+n^2\E|a_{0,12}|^p\right\}
 \le C_{p,K}\rho,
\end{equation}
where $a_{0,12}=a_0(X_1,X_2)$.  For the middle term, use
$b_0\le K/n$ and $\E b_0\le K\rho/n^2$; for the last, use $|a_0|\le K$.
The column sums satisfy the same estimate.

Conditionally on the second group, the $R_i$ are independent and centered.
The scalar form of Rosenthal's inequality gives the conditional bound
$C_p\{(\sum_i\E(R_i^2\mid Y))^{p/2}+
\sum_i\E(|R_i|^p\mid Y)\}$.
Jensen's inequality bounds the first term by
$C_p\E\{(\sum_iR_i^2)^{p/2}\mid Y\}$, and
$\sum_i|R_i|^p\le(\sum_iR_i^2)^{p/2}$ bounds the second.
Taking expectations, we obtain
\[
 \E\left|\sum_{i\in I}R_i\right|^p
 \le C_p\E\Big(\sum_{i\in I}R_i^2\Big)^{p/2}.
\]
This proves both bounds for the rectangular array.

For the original sample, assign each index independently to one of two
groups with probability $1/2$.  Retain only pairs joining different groups.
Given the observations, the expectation of the retained pair sum is half
the original sum.  The same holds coordinatewise for the vector of row
sums.  The coordinates in the two groups are the row and column sums of
the rectangular array.  Jensen's inequality and the preceding bounds prove
\eqref{eq:student-kernel-mom}.
\end{proof}

Choose an absolute constant $K\ge1$ larger than the bounds for $u$ and $v$ in
\eqref{eq:student-u-basic}--\eqref{eq:student-v-L2}, and let $t_0=(2K)^{-1}$.
For $|t|\le t_0$, define the measure
\begin{equation}
 P_t(\dd x)=(1+t u(x))P(\dd x).
\end{equation}
This is a probability measure because $\E u=0$, and its density lies between
$1/2$ and $3/2$.  Write $\E_t$ for expectation when all observations are
independent with law $P_t$.  The kernels and deterministic coefficients
remain as defined under $P$.  We need moment bounds uniform in $t$, as well as
bounds for three signed expectations.

\begin{lemma}\label{lem:student-tilted}
For every fixed $p\ge2$ and nonnegative integer $j$,
\begin{align}
 \sup_{|t|\le t_0}\E_t\{(1+|U|^j)(|V|^p+|\Delta|^p)\}
       &\le C_{p,j}\rho,\notag\\
 \sup_{|t|\le t_0}\E_t\{(1+|U|^j)\Delta_3\}
       &\le C_j\rho.
 \label{eq:student-weighted-moments}
\end{align}
Here $\Delta_3\ge0$ by \eqref{eq:student-C}.  Also,
\begin{equation}\label{eq:student-signed-means}
 \sup_{|t|\le t_0}
 \{ |\E_t V|+|\E_t(U\Delta_1)|+|\E_t(U\Delta_2)|\}\le C_0\rho.
\end{equation}
\end{lemma}

\begin{proof}
Consider also the probability measures
\begin{equation}
 \begin{split}
 P_{t,z}(\dd x)&=M(t,z)^{-1}(1+t u(x))e^{zu(x)}P(\dd x),\\
 M(t,z)&=\E\{(1+t u)e^{zu}\},\qquad |t|\le t_0,\quad |z|\le2.
 \end{split}
\end{equation}
For fixed $t,z$, let $Q=P_{t,z}$.  We write $\E_Q$ and $\E_{t,z}$
for expectation under the product law.  Set $w(x)=(1+t u(x))e^{zu(x)}$.
Since $u$ is bounded,
\begin{equation}\label{eq:student-tilt-weight}
 \begin{split}
 |w-1|&\le C_0|u|,\qquad |w-1-(t+z)u|\le C_0u^2,\\
 |M(t,z)-1|&\le C_0/n,\qquad M(t,z)^n\le C_0.
 \end{split}
\end{equation}
The density $dQ/dP$ is bounded above and below by positive absolute
constants, uniformly in $t,z$.

Let $\mu_u=\E_Q u$ and $w_0=u^2-1/n$.  From
\eqref{eq:student-u-basic}, \eqref{eq:student-tilt-weight}, and
$n^{-1/2}\le\rho$, we obtain
\begin{equation}\label{eq:student-tilt-u}
 \begin{split}
 |\mu_u|&\le C_0/n,\qquad
 |\E_Q w_0|\le C_0\rho/n,\\
 n\E_Q w_0^2+n\E_Q|u w_0|&\le C_0\rho.
 \end{split}
\end{equation}
For example, $n\E|u|^3\le C_0\rho$ and
$\E|u|\le n^{-1/2}$ bound the terms in
$n\E\{|u||u^2-1/n|\}$.  Rosenthal's inequalities show that all fixed
moments of $U$ and $\sum_i u_i^2$ are bounded under $Q$.  They also give
$\E_Q|\Delta_1|^p\le C_p\rho$ for $p\ge2$: apply
\eqref{eq:rosenthal-H} to $\sum_i(w_0(X_i)-\E_Qw_0)$, and use
\[
 2\sum_{i<j}u_i u_j=U^2-\sum_i u_i^2
\]
for the term with coefficient $(n-1)^{-1}$.

We next consider the pair kernels.  Define
\[
 q(x)=\E_Q v(x,X_1),\qquad \mu_v=\E_Q v(X_1,X_2),\qquad
 \bar q=q-\mu_v,
\]
and let
\[
 v_Q(x,y)=v(x,y)-q(x)-q(y)+\mu_v
\]
be the canonical projection under $Q$.  Canonicality under $P$ gives
\[
 q(x)=M(t,z)^{-1}\E\{v(x,X_1)(w(X_1)-1)\}.
\]
Hence the Cauchy--Schwarz inequality and \eqref{eq:student-v-L2} imply
\begin{equation}\label{eq:student-tilt-v-projections}
 \norm q_\infty\le C_0/n,
 \qquad n^3\E_Q q(X_1)^2\le C_0\rho.
\end{equation}
Indeed, $|q(x)|^2\le C_0 n^{-1}\E v(x,X_1)^2$.
For all fixed nonnegative integers $r,s$, we also have
\begin{equation}\label{eq:student-tilt-pair}
 \left|\E_Q\{u_1^r u_2^s v_{12}\}\right|
       \le C_{r,s}\rho/n^2.
\end{equation}
To prove this, write the expectation under $P$ with factors
$u^r w$ and $u^s w$.  When an exponent is zero, subtract $1$ from the
corresponding factor, using canonicality.  Each resulting factor has
absolute value at most a constant times $|u|$.  Equation
\eqref{eq:student-one-edge} proves \eqref{eq:student-tilt-pair}.
In particular, $|\mu_v|\le C_0\rho/n^2$.

The kernel $v_Q$ satisfies the assumptions of
Lemma~\ref{lem:student-kernel-mom}, with a different absolute $K$.
Indeed, uniformly in $t,z$,
\[
 \begin{gathered}
 \norm{v_Q}_\infty\le C_0,\qquad
 n^2\E_Q v_Q(X_1,X_2)^2\le C_0\rho,\\
 n\mathop{\rm ess\,sup}_{x}\E_Q v_Q(x,X_1)^2\le C_0.
 \end{gathered}
\]
The global bound follows from contraction of the canonical projection in
$L^2(Q^2)$ and the bounded marginal density $dQ/dP$.  The other two bounds
follow from \eqref{eq:student-v-L2}, \eqref{eq:student-tilt-v-projections},
and the bounds on $v$ and $\mu_v$.  By
\eqref{eq:student-tilt-v-projections}, the decomposition
\[
 V=\sum_{i<j}v_Q(X_i,X_j)+(n-1)\sum_i\bar q(X_i)
                  +\binom n2\mu_v
\]
gives $\E_Q|V|^p\le C_p\rho$ for $p\ge2$.
Indeed,
\[
 \norm{(n-1)\bar q}_\infty\le C_0,\qquad
 n(n-1)^2\E_Q\bar q^2\le C_0\rho.
\]
Rosenthal's inequality therefore bounds the $p$th moment of the centered
linear sum by $C_p\rho$, while the constant term is $O(\rho)$.

The same argument applies to
$v_1(x,y)=(u(x)+u(y))v(x,y)$.  Its global and conditional second moments
are bounded by constant multiples of those of $v$.  Define
\[
 q_1(x)=\E_Qv_1(x,X_1),\qquad \mu_1=\E_Qv_1(X_1,X_2).
\]
Then
\[
 q_1(x)=u(x)q(x)+M(t,z)^{-1}
              \E\{u(X_1)w(X_1)v(x,X_1)\}.
\]
The Cauchy--Schwarz inequality and \eqref{eq:student-tilt-v-projections} give
\[
 \norm{q_1}_\infty\le C_0/n,\qquad
 n^3\E_Qq_1(X_1)^2\le C_0\rho.
\]
Also, \eqref{eq:student-tilt-pair} gives
$|\mu_1|\le C_0\rho/n^2$.  Applying the preceding decomposition and
Rosenthal's inequality to $v_1$ yields
$\E_Q|\sum_{i<j}v_1(X_i,X_j)|^p\le C_p\rho$.
Since $|U|/n\le K$, formula \eqref{eq:student-B} gives
$\E_Q|\Delta_2|^p\le C_p\rho$.

For the quadratic denominator term, let
$r_i^Q=\sum_{j\ne i}v_Q(X_i,X_j)$.  The identity
\begin{equation}
 r_i=r_i^Q+(n-2)\bar q(X_i)+\sum_j\bar q(X_j)+(n-1)\mu_v
\end{equation}
shows that
$\E_Q(\sum_i r_i^2)^p\le C_p\rho$ for every $p\ge1$.
For the first term, apply Lemma~\ref{lem:student-kernel-mom}.
For the second, apply \eqref{eq:rosenthal-positive} to
$\sum_i(n\bar q(X_i))^2$: the summands are bounded and their total
expectation is at most $C_0\rho$.
For the third, \eqref{eq:rosenthal-H} gives
\[
 \E_Q\left\{n\Big(\sum_i\bar q(X_i)\Big)^2\right\}^{p}
 \le C_p n^{-p}\rho.
\]
Here the inequality is applied to the centered variables $n\bar q(X_i)$,
whose essential supremum and total variance satisfy the preceding bounds.
The last term has squared norm at most
$C_0 n^3\mu_v^2\le C_0\rho^2/n$.
Since $0\le \Delta_3\le C_0\sum_i r_i^2$, we obtain
\begin{equation}\label{eq:student-tilt-moments}
 \E_Q(|V|^p+|\Delta_1|^p+|\Delta_2|^p+|\Delta|^p)\le C_p\rho
 \quad(p\ge2),\qquad
 \E_Q \Delta_3^p\le C_p\rho\quad(p\ge1).
\end{equation}

For every nonnegative random variable $G$,
\begin{equation}
 \E_t(e^{zU}G)=M(t,z)^n\E_{t,z}G.
\end{equation}
The bound $|U|^j\le C_j(e^U+e^{-U})$, together with
\eqref{eq:student-tilt-weight} and \eqref{eq:student-tilt-moments} for
$z=1,-1$, proves \eqref{eq:student-weighted-moments}.

It remains to bound the signed expectations.  All fixed moments of
$U_{-12}=\sum_{i=3}^n u_i$ are bounded under $Q$.  By independence and
\eqref{eq:student-tilt-pair}, for every fixed integer $j\ge0$,
\begin{align*}
 \E_Q(U^jV)
 &=\binom n2\sum_{\ell=0}^j\binom j\ell
        \E_QU_{-12}^{j-\ell}
        \E_Q\{(u_1+u_2)^\ell v_{12}\},\\
 |\E_Q(U^jV)|&\le C_j\rho.
\end{align*}
Replacing $v_{12}$ by $(u_1+u_2)v_{12}$ gives the same estimate.
Equation~\eqref{eq:student-B} consequently yields
$|\E_Q(U\Delta_2)|\le C_0\rho$.
Finally,
\[
 \E_Q\left\{U\sum_iw_0(X_i)\right\}
 =n\E_Q(u w_0)+n(n-1)\mu_u\E_Qw_0.
\]
This is $O(\rho)$ by \eqref{eq:student-tilt-u}.
The expectation of the other term in $U\Delta_1$ has absolute value at most
\[
 \frac{C_0}{n}\E_Q\left\{|U|^3+|U|\sum_i u_i^2\right\}
 \le C_0/n\le C_0\rho.
\]
Taking $z=0$ proves \eqref{eq:student-signed-means}.
\end{proof}

\subsection{Completion of the $5/3$ argument}

We now bound the three quantities in \eqref{eq:finite-order-bound} for
the remainders $R_r$.  The third-moment term for the linear part is bounded in
\eqref{eq:student-u-basic}.  The preceding lemmas give the other two bounds.

\begin{lemma}
For $r=0,1$,
\begin{equation}\label{eq:student-R-L2}
 \E\Delta^2+\E\Delta^4\le C_0\rho,\qquad
 \E R_r^2\le C_0\rho,
\end{equation}
and
\begin{equation}\label{eq:student-tilted-mean-R}
 \sup_{|t|\le t_0}|\E_tR_r|\le C_0\rho.
\end{equation}
\end{lemma}

\begin{proof}
The bounds for $\Delta$ follow from
Lemma~\ref{lem:student-tilted} with $t=0$.  Write
\begin{equation}\label{eq:student-R-expanded}
 R_r=V-\tfrac12(U+V)\Delta+r(U+V)\Delta^2
          +d_n(1-\Delta/2+r\Delta^2).
\end{equation}
By \eqref{eq:student-weighted-moments},
$\E(U^2\Delta^2)+\E(U^2\Delta^4)\le C_0\rho$.
H\"older's inequality gives
\[
 \E(V^2\Delta^2)\le(\E V^4\,\E\Delta^4)^{1/2}\le C_0\rho,
 \qquad
 \E(V^2\Delta^4)\le(\E V^6)^{1/3}(\E\Delta^6)^{2/3}\le C_0\rho.
\]
Together with $|d_n|\le C_0\rho$, these estimates prove
$\E R_r^2\le C_0\rho$.

It remains to prove \eqref{eq:student-tilted-mean-R}.
The bounds in \eqref{eq:student-signed-means} and the estimate
$\E_t|U\Delta_3|\le C_0\rho$ give $|\E_t(U\Delta)|\le C_0\rho$.
The remaining products in \eqref{eq:student-R-expanded} satisfy
\begin{align*}
 \E_t|V\Delta|&\le(\E_tV^2\,\E_t\Delta^2)^{1/2}\le C_0\rho,\\
 \E_t(|U|\Delta^2)&\le C_0\rho,\\
 \E_t(|V|\Delta^2)&\le(\E_t|V|^3)^{1/3}
                          (\E_t|\Delta|^3)^{2/3}\le C_0\rho.
\end{align*}
The expectation of the term containing $d_n$ has absolute value at most $C_0\rho$.
This proves \eqref{eq:student-tilted-mean-R}.

\end{proof}

We now apply Corollary~\ref{cor:finite-order} and the sandwich lemma.

\begin{proof}[Proof of Theorem~\ref{thm:studentized}]
For $r=0,1$, write $Y_r=U+R_r$.  The finite-order expansion in
\eqref{eq:student-Yr} allows us to apply
Corollary~\ref{cor:finite-order}, with $d=10$ and an absolute constant $K$.
Equations \eqref{eq:student-u-basic} and \eqref{eq:student-R-L2}, together
with \eqref{eq:student-tilted-mean-R}, give
\[
 n\E|u_1|^3+\Var R_r+\sup_{|t|\le t_0}|\E_tR_r|\le C\rho.
\]
Hence
\begin{equation}\label{eq:student-Yr-normal}
 \norm{F_{Y_r}-\Phi}_\infty\le C\rho,\qquad r=0,1.
\end{equation}
By \eqref{eq:student-R-L2} and Markov's inequality,
\begin{equation}\label{eq:student-bad-denominator}
 \Pp(|\Delta|>1/2)\le4\E\Delta^2\le C\rho.
\end{equation}
Equations \eqref{eq:student-exact-ratio},
\eqref{eq:student-Yr-normal}, and \eqref{eq:student-bad-denominator},
together with Lemma~\ref{lem:student-sandwich}, prove the desired estimate
for $\sqrt n\,U_n^\flat/(2\widehat\sigma_\flat)$.

On $\mathcal E_n^c$, all pairwise values of $h^\flat$ equal the
corresponding values of $h$.  The numerator and the jackknife
denominator therefore agree with those of $\widehat S_n$.
Equation~\eqref{eq:student-censor-prob} transfers the bound to the original
statistic.  Also,
\[
 \Pp(\widehat\sigma_g=0)
 \le\Pp(\mathcal E_n)+\Pp(\Delta=-1)\le C\rho,
\]
so the convention used when the denominator is zero does not affect the result.
Restoring $\sigma_g$ proves \eqref{eq:studentized-main}.
\end{proof}

The truncation probability and the remainder variance are both bounded
by $C\rho$.  Theorem~\ref{thm:main} adds the latter directly to the normal
approximation error.  This gives the $5/3$ moment bound without a
logarithmic factor.

\subsection{Optimality}

We now establish optimality of the rate and moment exponents for bounds
uniform over the distribution and kernel.

\begin{proposition}\label{prop:student-sharp}
The rate $n^{-1/2}$ in Theorem~\ref{thm:studentized} is optimal, even when
$g$ is bounded and $\eta=0$.  With $\sigma_g=1$, each of the two moment
terms in \eqref{eq:studentized-main} can separately have the same order as
the error.  In each case, the other term is of smaller order and the error
tends to zero.
The following optimality assertions also hold.
\begin{enumerate}[label=\textup{(\roman*)}]
\item For every $1\le p<3$, there are laws and kernels, depending on $n$,
with $\sigma_g=1$, $\eta=0$ and $\sup_n\E|g_1|^p<\infty$, such that
\begin{equation}\label{eq:student-moment-sharp}
 \sqrt n\,\norm{F_{\widehat S_n}-\Phi}_\infty\longrightarrow\infty.
\end{equation}
\item For every $1\le q<5/3$, there are laws and kernels, depending on $n$,
with $g_1$ standard normal and $\sup_n\E|\eta_{12}|^q<\infty$, for which
\eqref{eq:student-moment-sharp} holds.
\end{enumerate}
\end{proposition}

\begin{proof}
First take $h(x,y)=x+y$, where the observations are independent symmetric
sign variables.
The Studentized statistic is the usual Studentized sample mean.  It has
an atom of mass at least $2^{-n}\binom n{\lfloor n/2\rfloor}$, obtained
when exactly $\lfloor n/2\rfloor$ observations equal $1$.  The sample
variance is positive on this event.  The atom argument in
Proposition~\ref{prop:main-sharp}, together with
Theorem~\ref{thm:studentized}, proves optimality of the rate.

We next use the collision example \eqref{eq:collision-definitions}.  Let
$A_n=\sqrt n(n-1)$ and define the order-two kernel
\begin{equation}
 \begin{split}
 h_n((z,j),(z',j'))&=z+z'+\eta_n(j,j'),\\
 \eta_n(j,j')&=A_n\{\1_{\{j=j'\}}-m_n^{-1}\}.
 \end{split}
\end{equation}
Its mean is zero, its first projection is $g(z,j)=z$, and its canonical
kernel is $\eta_n$.  In particular, $\sigma_g=1$.  The normalized
numerator is exactly
\[
 \frac{\sqrt n\,U_n}{2}=L_n+D_n-\lambda_n.
\]
Let $\bar Z=n^{-1}\sum_{i=1}^nZ_i$.  Conditionally on the labels $J_i$,
the vector of centered jackknife row means is a translate of
$\frac{n-2}{n-1}(Z_i-\bar Z)_{i=1}^n$.
This Gaussian vector is nondegenerate on the subspace of vectors whose
coordinates sum to zero.  Hence $\widehat\sigma_g>0$ almost surely, and the sign of
$\widehat S_n$ is that of its numerator.  Equation~\eqref{eq:rare-shift}
therefore gives
\begin{equation}\label{eq:collision-student-lower}
 F_{\widehat S_n}(0)-\tfrac12
       =c_*\lambda_n+O(\lambda_n^2).
\end{equation}

For every fixed $q\ge1$, the canonical moment is
\begin{equation}\label{eq:collision-student-moment}
 \E|\eta_{12}|^q
 =\frac{A_n^q}{m_n}
   \left\{(1-m_n^{-1})^q+(1-m_n^{-1})m_n^{1-q}\right\}
 \le\frac{2A_n^q}{m_n}.
\end{equation}
For $q>1$, this is asymptotic to $A_n^q/m_n$ as $m_n\to\infty$.
Taking $m_n=\lceil n^{9/4}\rceil$, we obtain
\[
 \frac{\E|\eta_{12}|^{5/3}}{\sqrt n}\sim2\lambda_n,
 \qquad \frac{\E|g_1|^3}{\sqrt n}=o(\lambda_n).
\]
The lower bound \eqref{eq:collision-student-lower} and
Theorem~\ref{thm:studentized} show that the canonical moment term has the
same order as the error.
For any $1\le q<5/3$, choose
\[
 \max\{2,3q/2\}<\xi<5/2,\qquad m_n=\lceil n^\xi\rceil.
\]
Then \eqref{eq:collision-student-moment} gives a uniformly bounded $q$th moment,
whereas $\sqrt n\lambda_n\sim\tfrac12n^{5/2-\xi}\to\infty$.
This proves the second assertion.

For the first assertion and the sharpness of the linear moment term,
fix $0<\alpha<1/2$ and let $B_{i,n}$ be independent Bernoulli variables
with success probability $\varepsilon_n=n^{-1-\alpha}$, independently
of standard normal variables $Z_i$.  Set $X_{i,n}=(Z_i,B_{i,n})$ and define
\begin{equation}
 \begin{split}
 g_n(z,b)&=\frac{z+\sqrt n(b-\varepsilon_n)}
                    {\sqrt{1+n\varepsilon_n(1-\varepsilon_n)}},\\
 h_n(x,y)&=g_n(x)+g_n(y).
 \end{split}
\end{equation}
The first projection has mean zero and variance one, and $\eta=0$.
The jackknife variance is again positive almost surely.  The sign of the
numerator is the sign of
\[
 n^{-1/2}\sum_iZ_i+K_n-\lambda_n,\qquad
 K_n=\sum_i B_{i,n},\qquad \lambda_n=n\varepsilon_n=n^{-\alpha}.
\]
Since $\E K_n=\lambda_n$ and
$\E\{K_n(K_n-1)\}\le\lambda_n^2$, the calculation leading to
\eqref{eq:rare-shift} applies to $K_n$.  Thus
\begin{equation}\label{eq:linear-student-lower}
 F_{\widehat S_n}(0)-\tfrac12
       =c_*\lambda_n+O(\lambda_n^2).
\end{equation}
Conditioning on $B_{1,n}$ gives
\[
 \E|g_n(Z_1,B_{1,n})|^3
       =(1+o(1))n^{3/2}\varepsilon_n+O(1)
       \sim n^{1/2-\alpha}.
\]
Consequently, the linear moment term in \eqref{eq:studentized-main} is
asymptotic to $\lambda_n$, and the theorem gives a matching upper bound.
Finally, for every fixed $p\ge1$,
\[
 \E|g_n(Z_1,B_{1,n})|^p
       \le C_p\{1+n^{p/2-1-\alpha}\}.
\]
For $1\le p<3$, choose
$\max\{0,p/2-1\}<\alpha<1/2$.  The $p$th moments are uniformly bounded, while
\eqref{eq:linear-student-lower} gives
$\sqrt n\,\norm{F_{\widehat S_n}-\Phi}_\infty
\ge c n^{1/2-\alpha}\to\infty$.
\end{proof}


\end{document}